\documentclass[a4paper,UKenglish,cleveref, autoref, thm-restate]{lipics-v2021}

\hideLIPIcs  

\usepackage[alignedforall]{lpform}
\usepackage{xparse}
\usepackage{braket,interval}
\usepackage{graphicx}
\usepackage[fleqn]{amsmath}
\usepackage{amsfonts}
\usepackage{amsthm}
\usepackage{amssymb}
\usepackage{amsbsy}
\usepackage{mathtools}
\usepackage{color} 
\usepackage{arydshln}
\usepackage{cite}

\def\ve#1{\mathchoice{\mbox{\boldmath$\displaystyle\bf#1$}}
	{\mbox{\boldmath$\textstyle\bf#1$}}
	{\mbox{\boldmath$\scriptstyle\bf#1$}}
	{\mbox{\boldmath$\scriptscriptstyle\bf#1$}}}

\newcommand{\Z}{\ensuremath{\mathbb{Z}}}
\newcommand{\R}{\ensuremath{\mathbb{R}}}
\newcommand{\Q}{\ensuremath{\mathbb{Q}}}

\newcommand{\OPT}
{\ensuremath {\mathsf {OPT}}}

\def\Orthant_j{{\mathcal O}_{j}}

\newcommand\veb{{\ve b}}

\newcommand\veg{{\ve g}}

\newcommand\veh{{\ve h}}
\newcommand\vel{{\ve \ell}}

\newcommand\vep{{\ve p}}

\newcommand\veu{{\ve u}}
\newcommand\vev{{\ve v}}
\newcommand\vew{{\ve w}}
\newcommand\vex{{\ve x}}
\newcommand\vey{{\ve y}}
\newcommand\vez{{\ve z}}

\newcommand{\OO}{{\mathcal{O}}}

\newtheorem*{th1}{Theorem 1}{\bf}{\it}
\newtheorem*{th2}{Theorem 3}{\bf}{\it}
\newtheorem*{th3}{Theorem 2}{\bf}{\it}

\NewDocumentCommand\overann{s O{\big} m m}{
	\IfBooleanTF{#1}{%
		\overset{\mathclap{#4}}{#3}%
	}{
		\overset{\mathclap{\substack{{#4} \\ #2\downarrow}}}{#3}%
	}
}
\NewDocumentCommand\underann{s O{\big} m m}{
	\IfBooleanTF{#1}{%
		\overset{\mathclap{#4}}{#3}%
	}{
		\underset{\mathclap{\substack{ #2\uparrow \\ {#4}}}}{#3}%
	}
}

\newcommand\nullity{\operatorname{nullity}}
\newcommand\rank{\operatorname{rank}}
\newcommand\img{\operatorname{img}}

\newcommand\supp{\operatorname{supp}}
\newcommand\iv[1]{[#1]}

\newcommand\nisupp{\operatorname{nisupp}}

\title{Approximation algorithms for integer programming with resource augmentation} 

\author{Hauke Brinkop}{Department of Computer Science, Kiel University, Germany}{hab@informatik.uni-kiel.de}{https://orcid.org/0000-0002-7791-2353}{Funded by the Deutsche Forschungsgemeinschaft (DFG,
   German Research Foundation) (No. 335406402)}

\author{Hua Chen}{School of Mathematical Sciences, Zhejiang University of Technology, China}{chenhua\_by@zju.edu.cn}{https://orcid.org/0000-0002-2315-5254}{Research supported in part by NSFC (No. 12401423)}

\author{Lin Chen}{College of Computer Science and Technology, Zhejiang University, China}{chenlin198662@gmail.com}{https://orcid.org/0000-0003-3909-4916}{Research supported in part by NSFC (No. 62572428)}

\author{Klaus Jansen}{Department of Computer Science, Kiel University, Germany}{kj@informatik.uni-kiel.de}{https://orcid.org/0000-0001-8358-6796}{Funded by the Deutsche Forschungsgemeinschaft (DFG, German Research Foundation) (No. 
528381760)}

\author{Guochuan Zhang}{College of Computer Science and Technology, Zhejiang University, China}{zgc@zju.edu.cn}{https://orcid.org/0000-0003-1947-7872}{Research supported in part by NSFC (No. 12271477)}

\authorrunning{H. Brinkop et al.} 

\Copyright{Hauke Brinkop, Hua Chen, Lin Chen, Klaus Jansen, and Guochuan Zhang} 

\ccsdesc[100]{Theory of computation $\rightarrow$ Design and analysis of algorithms} 

\keywords{Approximation algorithms, Resource augmentation, Integer programs, $n$-fold IPs} 

\category{} 

\nolinenumbers 

\EventEditors{John Q. Open and Joan R. Access}
\EventNoEds{2}
\EventLongTitle{43rd International Symposium on Theoretical Aspects of Computer Science (STACS 2026)}
\EventShortTitle{STACS 2026}
\EventAcronym{STACS}
\EventYear{2026}
\EventDate{December 24--27, 2016}
\EventLocation{Little Whinging, United Kingdom}
\EventLogo{}
\SeriesVolume{42}
\ArticleNo{26}

\begin{document}
\maketitle

\begin{abstract}
Solving a general integer program (IP) is NP-hard. The classic algorithm [Papadimitriou, J.ACM '81] for IPs has a running time $n^{\OO(m)}(m\cdot\max\{\Delta,\|\veb\|_{\infty}\})^{\OO(m^2)}$, where $m$ is the number of constraints, $n$ is the number of variables, and $\Delta$ and $\|\veb\|_{\infty}$ are, respectively, the largest absolute values among the entries in the constraint matrix and the right-hand side vector of the constraint.  The running time is exponential in $m$, and becomes pseudo-polynomial if $m$ is a constant. In recent years, there has been extensive research on FPT (fixed parameter tractable) algorithms for the so-called $n$-fold IPs, which may possess a large  number of constraints, but the constraint matrix satisfies a specific block structure. It is remarkable that these FPT algorithms take as parameters $\Delta$ and the number of rows and columns of some small submatrices. If $\Delta$ is not treated as a parameter, then the running time becomes pseudo-polynomial even if all the other parameters are taken as constants.   

This paper explores the trade-off between time and accuracy in solving an IP. We show that, for arbitrary small $\varepsilon>0$, there exists an algorithm for IPs with $m$ constraints that runs in ${f(m,\varepsilon)}\cdot\textnormal{poly}(|I|)$ time, and 
returns a near-feasible solution that violates the constraints by at most $\varepsilon\Delta$. 
Furthermore, for $n$-fold IPs, we establish a similar result---our algorithm runs in time that depends on the number of rows and columns of small submatrices together with $1/\varepsilon$, and returns a solution that slightly violates the constraints. Meanwhile, both solutions guarantee that their objective values are no worse than the corresponding optimal objective values satisfying the constraints.
As applications, our results can be used to obtain additive approximation schemes for multidimensional knapsack as well as scheduling. 

 
\end{abstract}

\clearpage

\section{Introduction}\label{sec:typesetting-summary}

Integer programs (IPs) find widespread applications in many combinatorial optimization problems. Generally, an integer program can be formulated as follows:
\begin{eqnarray}\label{ILP:1}
 &\min\set{\vew\vex : {H} \vex=\veb, \, \vel\le \vex\le \veu,\, \vex\in \Z^{n} }, 
\end{eqnarray}
where $H\in \Z^{m\times n}$ is the constraint matrix, $\veb\in \Z^{m}$ is the vector on the right-hand side of the equality constraint, $\vew\in \Z^{n}$ is the coefficient vector of the objective function, and $\vel,\veu\in \Z^{n}$ are the lower and upper bounds of the variables, respectively. 
Karp~\cite{karp1972reducibility} showed that solving IP~(\ref{ILP:1}) is NP-hard in 1972. Subsequent studies thus focus on finding efficient algorithms for IPs where the constraint matrix has some specific structure.  

Lenstra~\cite{lenstra1983integer} studied IPs with few variables and presented a $2^{\OO(n^3)}\cdot\text{poly}(|I|)$-time algorithm, where $\textnormal{poly}(|I|)$ denotes a polynomial in the input length $|I|$. The running time has been improved in several subsequent works~\cite{kannan1987minkowski, dadush2011enumerative,dadush2012integer,reis2023subspace}. Very recently, Reis and Rothvoss~\cite{reis2023subspace} proved that IP~(\ref{ILP:1}) can be solved in $(\log n)^{\OO(n)}\cdot\text{poly}(|I|)$-time.   


Papadimitriou~\cite{papadimitriou1981complexity} considered IPs with few constraints and provided an algorithm running in $n^{\OO(m)}(m\cdot\max\{\Delta,\|\veb\|_{\infty}\})^{\OO(m^2)}$-time, where $\Delta\coloneqq\|H\|_{\infty}$  represents the largest absolute value among all entries of $H$, and $\|\veb\|_{\infty}$ represents the largest absolute value among all entries of $\veb$. The running time was later improved by Eisenbrand and Weismantel~\cite{eisenbrand2019proximity}, and Jansen and Rohwedder~\cite{jansen2023integer}. In particular, for the special case when $\vel=\ve0$ and $\veu=\ve\infty$, IP~(\ref{ILP:1}) can be solved in $\OO(\sqrt{m}\Delta)^{2m}\cdot\log (\|\veb\|_\infty)+\OO(nm)$ time~\cite{jansen2023integer}, and lower bounds on the running times have also been studied by~\cite{knop2020tight,jansen2023integer,fomin2023optimality}.

The results mentioned above assume that the constraint matrix $H$ has either few rows or few columns. There has also been extensive research on IPs where the number of rows and columns of $H$ can be large but $H$ possesses some block structure; see, e.g.,~\cite{de2008n,hemmecke2013n,hemmecke2003decomposition,aschenbrenner2007finiteness,cslovjecsek2021efficient,eisenbrand2019algorithmic,jansen2021double,klein2021complexity,koutecky2018parameterized,cslovjecsek2024parameterized}. One most widely studied structure is the so-called $n$-fold, where $H$ consists of small submatrices $A^i$, $D^i$ of the following form
\begin{eqnarray*}\label{eq:matrix}
{H}\coloneqq
\begin{pmatrix}
  D^1 & D^2 & \cdots & D^n \\
  A^1 & 0  &   & 0  \\
  0  & A^2 &   & 0  \\
\vdots   &   & \ddots &  \\
 0  & 0  &   & A^n
\end{pmatrix}
,
\end{eqnarray*}
where $A^i\in \Z^{s_A\times t}$, and $D^i\in \Z^{s_D\times t}$ for all $i$. Many parameterized algorithms for $n$-fold IPs have been developed~\cite{hemmecke2013n,altmanova2019evaluating,cslovjecsek2021block,eisenbrand2018faster,eisenbrand2019algorithmic,jansen2022empowering,jansen2019near,cslovjecsek2024parameterized}. In particular,
Cslovjecsek et al.~\cite{cslovjecsek2021block} presented an algorithm of running time $2^{\OO(s^2_As_D)}(s_Ds_A\Delta)^{\OO(s_A^2+s_As_D^2)} \cdot (nt)^{1+o(1)}$, where $\Delta=\|H\|_{\infty}$, which is currently the best known FPT algorithm parameterized by $\Delta, s_A,s_D$. Here, an FPT algorithm refers to an algorithm with running time $g(k)\cdot\textnormal{poly}(|I|)$, where $g$ is a computable function, $|I|$ is the input length, and $k$ is a parameter independent of $|I|$.
    Note that their approach relies on proximity, which is (necessarily) dependent on $\Delta$ and thus not usable for our use-case; if there was a proximity bound of the form $f(s_A,s_D) \cdot \textnormal{poly}(|I|)$, this would imply an $n$-fold algorithm (using the approach in \cite{eisenbrand2019proximity}) that is FPT with parameters $s_A,s_D$. 
    Finally, this would imply a polynomial time algorithm for subset sum~\cite{lassota2025parameterized}. 

We see that the running time is either exponential in the number of variables or dependent on $\Delta$ and is thus pseudo-polynomial. Such a situation is unavoidable if we are looking for an exact algorithm that solves IP~\eqref{ILP:1} optimally, since all algorithms for bounded number of constraints must also depend on $\Delta$ as it could otherwise solve subset sum in polynomial time. However, a polynomial time algorithm may be possible if we are allowed to solve IP~\eqref{ILP:1} approximately in the sense that the constraints can be violated slightly.

One most well-known example is Knapsack, where $H$ contains one row, i.e., $m=1$, $\vel=\ve0$, $\veu=\ve1$, and all entries of $\vew$ and $H$ are nonnegative\footnote{Knapsack requires $H\vex\le \veb$ instead of $H\vex= \veb$, but we can add dummy items of 0 profit and unit weight to enforce the equation.}. Knapsack admits an FPTAS whose running time is polynomial in the input length $|I|$ and $1/\varepsilon$ instead of $\Delta$ (see, e.g.,~\cite{ibarra1975fast,chan2018approximation,jin2019improved,deng2023approximating,chen2024nearly,mao20241}). Alternatively, it admits an algorithm of the same running time that returns a solution whose objective value is at least \OPT\ (where \OPT\ denotes the optimal objective value satisfying the constraint), while violating the constraint by a factor of $1+\varepsilon$. 

Another example is due to Dadush~\cite{dadush2014randomized}, who presented an algorithm of running time $2^{\OO(n)}{(1/\varepsilon^2)}^n$ which either asserts that the convex body $K$ described by the constraints does not contain any integer points, or finds
an integer point in the body stemming from $K$ scaled by $1+\varepsilon$ from its center of gravity.

Motivated by the above two examples, this paper aims to explore approximation algorithms for IPs. We are interested in an algorithm that returns a near-feasible solution that may violate the constraints by a factor of $1+\varepsilon$, but runs in polynomial time in the input. Such relaxation of constraints is also referred to as \emph{resource augmentation}, which has received much attention in various combinatorial optimization problems, including  Knapsack~\cite{bringmann2022faster,iwama2010online}, Subset Sum~\cite{chen2024approximating,mucha2019subquadratic}, Square Packing~\cite{fishkin2005packing}, etc.


\subsection{Our contribution}

Our first result is an approximation scheme for IPs with few constraints. 

 \begin{theorem}\label{conj:IP-general}
 	Given is an integer program $\min\{\vew\vex: H \vex=\veb, \, \vel\le \vex\le \veu,\, \vex\in \Z^{n} \}$ 
    with optimal objective value $\OPT$, where $H\in\Q^{m\times n}$.  
 	Then for arbitrary small $\varepsilon>0$, there exists an algorithm of running time ${f(m,\varepsilon)}\cdot\textnormal{poly}(|I|)$ which returns a near-feasible solution $\tilde{\vex}$ such that $\tilde{\vex}\in \set { \vex : \|H\vex-\veb\|_{\infty}\le\varepsilon\Delta, \, \vel\le \vex\le \veu,\, \vex\in \Z^{n} }$, and $\vew\tilde{\vex}\le \OPT$, where $\Delta=\|H\|_{\infty}$. 
 \end{theorem}
In particular, the precise running time in Theorem~\ref{conj:IP-general} is $2^{(\frac{m}{\varepsilon})^{\OO(m)}}\cdot \textnormal{poly}(|I|)$.
We remark that \OPT\, always refers to the optimal objective value {\bf without} violating the constraints. 

Theorem~\ref{conj:IP-general} extends Papadimitriou's algorithm in the direction of approximation. When $\varepsilon<1/\Delta$, for example, $\varepsilon=1/(2\Delta)$, Theorem~\ref{conj:IP-general} implies an exact FPT algorithm to solve IP~\eqref{ILP:1} with running time  $2^{{(m\Delta)}^{\OO(m)}}\cdot \textnormal{poly}(|I|)$.
A special case of IP~\eqref{ILP:1} is the classic multidimensional knapsack problem, where all the entries in $\vew, H, \vel,\veu$ are nonnegative. Thus Theorem~\ref{conj:IP-general} gives an EPTAS for the multidimensional knapsack problem with resource augmentation (i.e., by slightly relaxing the multidimensional knapsack constraints). Approximation schemes with resource augmentation (or weak approximation schemes) for the knapsack problem have been studied extensively in recent years~\cite{mucha2019subquadratic,bringmann2022faster,chen2025weakly}. Note that if the constraints cannot be relaxed, then there exists a PTAS of running time $O(n^{\lceil m/\epsilon\rceil})$ for $m$-dimensional knapsack~\cite{caprara2000approximation}, and such a running time is essentially the best possible assuming ETH (Exponential Time Hypothesis)~\cite{doron2024fine}. In particular, 2-dimensional knapsack does not admit an EPTAS assuming W[1] $\neq$ FPT~\cite{kulik2010there}. Moreover, it does not admit a PTAS of running time $n^{o(1/\epsilon)}$ assuming ETH~\cite{jansen2016bounding}.

We then extend our result to $n$-fold IPs and obtain the following.

\begin{theorem}\label{conj-coro:nfold-additive}
Given is an integer program $\min\{\vew\vex: \sum_{i=1}^nD^i\vex^i=\veb^0, A^i\vex^i= \veb^i,1\le i\le n,\, \ve0\le \vex\le \veu, \vex\in \Z^{nt} \}$ with optimal objective value $\OPT$, where $A^i\in\Q^{s_A\times t_A}_{\ge0}$, $D^i\in\Q^{s_D\times t_D}_{\ge 0}$, and $t_A=t_D=t$.   
	Then for arbitrary small
 $\varepsilon>0$, there exists an algorithm of running time ${f(s_A,s_D,t,\varepsilon)}\cdot\textnormal{poly}(|I|)$ which returns a near-feasible solution $\tilde{\vex}$ such that $\tilde{\vex}\in\{\vex: (1-\varepsilon)\veb^0\le\sum_{i=1}^nD^i\vex^i\le (1+\varepsilon)\veb^0,\, (1-\varepsilon)\veb^i\le A^i\vex^i\le (1+\varepsilon)\veb^i, 1\le i\le n,\, \ve0\le \vex\le\veu,\, \vex\in \Z^{nt} \}$, and $\vew\tilde{\vex}\le \OPT$. 
\end{theorem}


In particular, the running time is $2^{{2^{{({s_D}/{\varepsilon})^{\OO(s_D)}\cdot ({s_At}/{\varepsilon})^{\OO(t)}}}}}\cdot \textnormal{poly}(|I|)$.


Theorem~\ref{conj-coro:nfold-additive} is a bit weak in the sense that it requires all the entries of the submatrices $A^i,D^i$ as well as the variables to be nonnegative, and it gives an additive error of $\varepsilon\veb^i$ instead of $\varepsilon \Delta$. In many combinatorial optimization problems where $n$-fold IPs are applicable, all the input parameters are indeed nonnegative and $(1+\varepsilon)\veb^i$ yields a standard multiplicative $(1+\varepsilon)$-factor. Nevertheless, we show that it is possible to get rid of such a weakness if the local constraint $A^i\vex^i= \veb^i$ has few solutions, as implied by the following Theorem~\ref{conj:nfold-additive}.



\begin{theorem}\label{conj:nfold-additive}
Given is an integer program  $\min\{\vew\vex: \sum_{i=1}^nD^i\vex^i=\veb^0, \vex^i\in\mathcal{P}^i, 1\le i\le n,\, \vex\in \Z^{nt} \}$ with optimal objective value $\OPT$, where $D^i\in\Q^{s\times t}$ and $\mathcal{P}^i$ is an arbitrary set of integer vectors.  Let  $\kappa=\max_{\veu\in \cup_{i=1}^n\mathcal{P}^i}\|\veu\|_\infty$, i.e., $\kappa$ is the largest $\ell_\infty$-norm among all integer vectors in $\cup_{i=1}^n\mathcal{P}^i$. 
	Then for arbitrary small
$\varepsilon>0$, there exists an algorithm of running time ${f(s,t,\kappa,\varepsilon)}\cdot\textnormal{poly}(|I|)$ which returns a near-feasible solution $\tilde{\vex}$ such that $\tilde{\vex}\in\{\vex: \|\sum_{i=1}^nD^i\vex^i-\veb^0\|_{\infty}\le \varepsilon\Delta, \,\vex^i\in\mathcal{P}^i,1\le i\le n,\, \vex\in \Z^{nt} \}$, and $\vew\tilde{\vex}\le \OPT$, where $\Delta=\max_{i\in [n]}\|D^i\|_{\infty}$. 
\end{theorem}

In particular, the running time is $2^{(\frac{s\kappa^t}{\varepsilon})^{\OO(s\kappa^t)}}\cdot \textnormal{poly}(|I|)$.

Theorem~\ref{conj-coro:nfold-additive} and Theorem~\ref{conj:nfold-additive} can be viewed as FPT approximation schemes for $n$-fold IPs with different parameters. 
One application is to provide additive approximation schemes for the unrelated machine scheduling problem $Rm||C_{\max}$: 
Given are $n$ jobs and $m$ machines, with the processing time of job $i$ on machine $h$ as $p_{ih}$. The goal is to assign jobs to machines such that the makespan is at most some target value $C_{\max}$. Let $x^i_h\in\{0,1\}$ denote whether job $i$ is assigned to machine $h$, i.e., if job $i$ is assigned to machine $h$, then $x^i_h=1$; otherwise, $x^i_h=0$. Then the natural IP formulation is the following  feasibility test:
 \begin{eqnarray*}
 &\Big\{\vex: \sum_{i=1}^np_{ih}x_{h}^i\le C_{\max}, \, \forall 1\le h\le m, \,\vex^i\in\mathcal{P}^i,1\le i\le n,\, \vex\in \Z^{nt} \Big\}, 
 \end{eqnarray*}
 where polytope $\mathcal{P}^i\coloneqq\{\vex^i:x^i_1+x^i_2+\cdots+x^i_m=1, x^i_h\in\{0,1\}\}$ implies that exactly one $x^i_h$ is 1. Observe that $D^i\in \Z^{m\times m}$ and $\kappa=1$, so Theorem~\ref{conj:nfold-additive} implies a PTAS of running time $f(m,\varepsilon)\cdot n^{\OO(1)}$ for finding a feasible schedule $\tilde{\vex}$ such that the makespan is at most $C_{\max}+\varepsilon \cdot\max_{i\in[n],h\in[m]}p_{ih}$.  In comparison, the current best-known PTAS for unrelated machine scheduling runs in $2^{\OO(m\log(m/\delta))}$-time~\cite{jansen2010scheduling}, and returns a solution with makespan $(1+\delta)C_{\max}$. Using this algorithm to derive an additive PTAS would require setting $\delta=\OO(\varepsilon/n)$, yielding a non-polynomial running time. 
 
 Theorem~\ref{conj:nfold-additive} can also be used to obtain PTASes for other variants of scheduling problems. Consider the generalized assignment problem, where the input is the same as the above-mentioned unrelated machine scheduling problem, except that there is a cost $c_{ih}$ for scheduling job $i$ on machine $h$.  The goal is a bi-criteria study on the makespan $C_{\max}$ and the total scheduling cost $T$. Taking the total scheduling cost as a constraint, a classic result of Shmoys and Tardos~\cite{shmoys1993approximation} gave a 2-approximation algorithm.  
 Angel et al.~\cite{angel2001fptas} proposed a PTAS for $Rm|c_{ih}|C_{\max}$ of running time $\OO(n(n/\varepsilon)^{m})$. Very recently Li et al.~\cite{weidong} showed a PTAS for identical machine scheduling $Pm|c_{ih}|C_{\max}$ of an improved running time $n(m/\varepsilon)^{\OO(m)}$, which is an EPTAS when taking $m$ as a parameter. Using Theorem~\ref{conj:nfold-additive}, an EPTAS for unrelated machines $Rm|c_{ih}|C_{\max}$ of running time $2^{(\frac{m}{\varepsilon})^{\OO(m)}}\cdot n^{\OO(1)}$ directly follows.
It is worth mentioning that researchers have studied the bi-objective for various other models (e.g., different machine environments, specific cost functions like $c_{ih}=a_i\cdot b_h$ or $c_{ih}=a_i+b_h$, etc.)~\cite{li2024approximation,lee2014fast}, and we believe Theorem~\ref{conj:nfold-additive} may also be applicable to some of them.


\subsection{Technical overview}  We first briefly discuss our technique for proving Theorem~\ref{conj:IP-general}.
Let $H=(\veh_1,\veh_2,\dots,\veh_n)$, whereas $H\vex=\sum_{j=1}^n\veh_jx_j$. Observe that $\veh_j\in [-\Delta,\Delta]^m$. If $n$ is small, say $n=m^{\OO(1)}$, then Lenstra's algorithm~\cite{lenstra1983integer} works directly and solves the IP in $n^{\OO(n)}=2^{m^{\OO(1)}}$ time, concluding Theorem~\ref{conj:IP-general}. So we may assume that $n$ is sufficiently large, which means there are sufficiently many $\veh_j$'s within the box $[-\Delta,\Delta]^m$. Consequently, many $\veh_j$'s are very similar to each other in the sense that their difference lies in a small hypercube of side length $\OO(\varepsilon \Delta) $. Now suppose $\veh_1$ and $\veh_2$ are similar. To approximately solve IP~\eqref{ILP:1}, we observe that it is not necessary to keep two integer variables $x_1$ and $x_2$ for $\veh_1$ and $\veh_2$ respectively. Instead, we may introduce a hyper-integer variable $y$ that represents $x_1+x_2$, and relax $x_1$ and $x_2$ to fractional variables. By doing so, we transform the original IP to a mixed integer program (MIP) with a constant number of integer variables, which is solvable in polynomial time. Then we round fractional variables to integer variables by maintaining the overall error incurred. The challenge here is that, to bound the overall error, we need to make sure that among all the fractional variables, only a few of them can take fractional values. Interestingly, while these fractional variables need to satisfy a large number of constraints, the specific structure of the constraint matrix allows us to conclude that very few variables may take a fractional value in a vertex solution. Towards showing this, we introduce some necessary notions and lemmas regarding vertex solutions in preliminaries.  

The proof strategy for Theorem~\ref{conj-coro:nfold-additive} can be viewed as a combination of that for Theorem~\ref{conj:IP-general} and Theorem~\ref{conj:nfold-additive}, so we briefly discuss our technique for proving Theorem~\ref{conj:nfold-additive}. For simplicity, let us assume that $\mathcal{P}^i=\mathcal{P}=\{\vep_1,\vep_2,\dots,\vep_{\tau}\}$ for all $i$, where $\tau\coloneqq\max_{i\in [n]}|\mathcal{P}^i|=\kappa^t$. Hence, we see that the vector $D^i\vex^i$ is essentially one column out of the precomputed matrix $\mathcal{D}^i:=\{D^i\vep_1,D^i\vep_2,\dots,D^i\vep_{\tau}\}$. 
If $D^i$ and $D^j$ (and hence $\mathcal{D}^i$ and $\mathcal{D}^j$) are close to each other, then we do not want to have separate integer variables $\vex^i$ and $\vex^j$ for them, instead, we want to introduce a hyper-integer variable $\vey$ representing $\vex^i+\vex^j$, and relax $\vex^i$ and $\vex^j$ to fractional variables. It turns out that we need to go a bit further: We need $\tau$ hyper-integer variables $\vey_1$ to $\vey_{\tau}$ since there are $\tau$ choices out of $\mathcal{P}^i$. Going a bit further does not affect much when we transform the original IP into an MIP with a constant number of integer variables; however, it causes the established MIP to contain way more constraints. Our next goal is still to argue that in a vertex solution of the LP (induced by the MIP by fixing the integer variables), only a few variables may take a fractional value. To achieve this, we observe that the constraint matrix of the LP can be decomposed into two parts (two submatrices), one having a small rank, and the other having a large rank but also having a nice diagonal block structure. We show that the property of the vertex solution to such a two-part LP can be studied by arguing on the two parts separately.   

\subsubsection{Comparison to some relevant prior works.}~\label{subsec:compare} Our algorithmic results on block-structured IPs (Theorem~\ref{conj-coro:nfold-additive} and~\ref{conj:nfold-additive}) are closely related to their FPT algorithms. In addition to $n$-fold IPs, researchers have also studied FPT algorithms for 2-stage stochastic IPs and 4-block $n$-fold IPs, whose constraint matrices can be written as:
\begin{eqnarray*}
{H_{\textnormal{2-stage}}}:=
\begin{pmatrix}
B^1 & A^1 & 0  &   & 0  \\
B^2 & 0  & A^2 &   & 0  \\
\vdots &   &   & \ddots &   \\
B^n & 0  & 0  &   & A^n
\end{pmatrix}
, \enspace
 {H_{\textnormal{4-block}}}:=
\begin{pmatrix}
C & D^1 & D^2 & \cdots & D^n \\
B^1 & A^1 & 0  &   & 0  \\
B^2 & 0  & A^2 &   & 0  \\
\vdots &   &   & \ddots &   \\
B^n & 0  & 0  &   & A^n
\end{pmatrix}
.
\end{eqnarray*}
These FPT algorithms typically find an exact solution within a running time of $f(\Delta)\cdot\textnormal{poly}(|I|)$ for some computable function $f$ (see Section~\ref{subsec:related} for detailed related works). By taking $\varepsilon$ to be sufficiently small (e.g., $\varepsilon<1/\Delta$ in Theorem~\ref{conj:IP-general} and Theorem~\ref{conj:nfold-additive}), our approximation algorithms can return an exact solution, albeit that they may have a larger running time compared with the existing FPT algorithms. The merit of our approximation algorithms is that they offer a more flexible control over the running time vs. accuracy. 

Very recently, Eisenbrand and Rothvoss~\cite{eisenbrand2025parameterized} studied $4$-block $n$-fold IPs and obtained a result of a similar flavor. They showed that, in time $f(\max_i\|A^i\|_\infty)\cdot \textnormal{poly}(|I|)$, one can find a solution $(\vex^0,\vex^1,\dots,\vex^n)$ whose objective value is no worse than $\OPT$, and which satisfies all constraints except the first one, i.e., $C\vex^0+\sum_{i=1}^nD^i\vex^i\neq \veb^0$. In particular,  the error can be bounded as $\|C\vex^0+\sum_{i=1}^nD^i\vex^i-\veb^0\|_\infty\le g(\max_i\|D^i\|_\infty)$ for some computable function $g$. As $n$-fold IPs are special cases of $4$-block $n$-fold IPs, such a result also holds for $n$-fold IPs. Such a result is parallel to our work since the running time of the algorithm provided by Eisenbrand and Rothvoss depends on $\Delta$, and its error can be way larger than $\Delta$. 

In terms of techniques, our algorithms need to round a fractional solution to an integral solution. During the rounding procedure, we will need to reduce the number of variables that take a fractional value, and this resembles the work of Cslovjecsek et al.~\cite{cslovjecsek2021block}. In particular, Cslovjecsek et al. considered an LP of the form $A_1\vex^1+\cdots+ A_n\vex^n= \veb$, where $\vex^i\in Q_i$ and each $Q_i\subseteq \R^{t_i}_{\ge 0}$ is an integral polyhedron. They showed that in a vertex solution, all but $m$ of the $\vex^i$'s are vertices of the $Q_i$'s, respectively, and thus all but $m$ of the $\vex^i$'s are integral, where $m$ is the number of rows of $A_i$'s. Their result applies directly to $n$-fold IPs where each $A_i$ is a small submatrix of fixed size. Our rounding procedure also deals with an LP of the form $A_1\vex^1+\cdots+ A_n\vex^n= \veb$; however, each $A_i$ will have a huge size. The huge size is mainly due to a configuration-LP type of reformulation of the input IP. While the result of Cslovjecsek et al. implies that in a vertex solution only a few $\vex^i$'s are integral, one $\vex^i$ may still contain too many fractional variables. We will need to leverage the special structure of the $A_i$'s to reduce the number of fractional variables (See Section~\ref{sec:3.2} for details). Overall, while we are following the general framework of arguing on the number of fractional variables in a vertex solution to a specific $n$-fold-like LP, we are not aware of a prior result that is applicable directly.

\subsection{Related work}\label{subsec:related}
The approximation algorithms of IPs are rare in the literature. Dadush~\cite{dadush2014randomized} studied \emph{approximate integer programming} in 2014. Later, Dadush et al.~\cite{dadush2024approximate} investigated how to reduce the exact integer programming
problem to the approximate version. In addition,
Shmoys and Swamy~\cite{shmoys2004stochastic} showed that one could
derive approximation algorithms for most of the stochastic IPs considered in~\cite{dye2003stochastic, gupta2004boosted,immorlica2004costs, ravi2006hedging} by adopting a natural LP rounding approach. Swamy and Shmoys~\cite{swamy2005sampling} gave the first approximation algorithms for a variety of $k$-stage generalizations of basic combinatorial optimization problems including the set cover, vertex cover, multicut on trees, facility location, and multicommodity flow problems. 

We give additive approximation schemes for IPs, which yield additive approximation schemes for several scheduling problems. While most of the existing approximation schemes for scheduling problems give a multiplicative factor of $1+\varepsilon$, additive approximation schemes have been studied very recently by Buchem et al.~\cite{buchem2021additive}, who presented an algorithm for $Pm||C_{\max}$ with a running time of $m^2n^{\OO(1/\varepsilon)}$ 
that computes a solution with makespan at most $\OPT + \varepsilon\cdot \max_{j\in[n]}p_{j}$. It is interesting to study additive approximation schemes for other classic scheduling problems.

In addition to scheduling problems, $n$-fold IPs have also been used to model a variety of other combinatorial optimization problems, based on which FPT algorithms are obtained; see, e.g.,~\cite{knop2020combinatorial,hemmecke2013n,de2008n,hemmecke2013n,hemmecke2011n,knop2019multitype,onn2011theory,knop2023high,knop2018scheduling,jansen2021total,jansen2022empowering}. We believe that the study of approximation schemes for $n$-fold type of IPs can also facilitate the design of approximation schemes for relevant problems.

FPT algorithms for IPs with respect to other structure parameters have also been studied extensively, such as 2-stage stochastic IPs~\cite{aschenbrenner2007finiteness,cslovjecsek2021efficient,eisenbrand2019algorithmic,jansen2021double,klein2021complexity,koutecky2018parameterized,hemmecke2003decomposition} and $4$-block $n$-fold IPs~\cite{hemmecke2014graver,chen2020new,chen2022blockstructured,chen2024fpt,oertel2024colorful}, as well as other structured IPs~\cite{brand2021parameterized,eiben2019integer,chan2022matrices,eisenbrand2025parameterized}. 
Particularly, Cslovjecsek et al.~\cite{cslovjecsek2024parameterized} proved that the feasibility problems for 2-stage stochastic IPs can be solved in time $f(\gamma_1,\max_i\|A^i\|_\infty)\cdot |I|$ for some computable function $f$, where $\gamma_1$ is an upper bound of the dimensions of all blocks $A^i,B^i$. They also demonstrated that the optimization problems for $n$-fold IPs with $D^1=D^2=\cdots=D^n$ can be solved in time $f(\gamma_2,\max_i\|A^i\|_\infty)\cdot \textnormal{poly}(|I|)$, where $\gamma_2$ is an upper bound of the dimensions of all blocks $A^i,D^i$. 
Very recently, Eisenbrand and Rothvoss~\cite{eisenbrand2025parameterized} solved an open problem raised by Cslovjecsek et al.~\cite{cslovjecsek2024parameterized} and showed that the optimization problems for 2-stage stochastic IPs can be solved in time $f(\gamma_1,\max_i\|A^i\|_\infty)\cdot |I|$.

It is worth mentioning that block-structured IPs with fractional variables have also been explored in the literature. In particular, Brand et al.~\cite{brand2021parameterized} introduced the concept of {\it fractionality},  which denotes the minimum over the maxima of denominators in optimal solutions of an MILP instance. They showed that the fractionality of $n$-fold and 2-stage stochastic mixed-IPs is FPT-bounded.

\subsection{Preliminaries}

\noindent\textbf{Notations.} We write column vectors in boldface, e.g., $\vex, \vey$, and their entries in normal font, e.g., $x_j, y_j$. Let $[i]\coloneqq\set{1,2,\dots,i}$. For a vector or a matrix, let $\|\cdot\|_{\infty}$ denote the maximum absolute value of its elements. 
Let $\textnormal{poly}(|I|)$ denote a polynomial in the input length $|I|$, $f$ be a computable function, and \OPT\, denote the optimal objective value. 

The support of a vector $\vex$ is the set of indices $I = \supp(\vex)$ such that $x_j = 0 \iff j \notin I$.
Let $V$ and $W$ be vector spaces and $\phi \colon V \to W$.
The kernel of $\phi$, denoted $\ker(\phi)$, is the set of vectors mapping to zero, i.e., 
$\ker(\phi) = \set { x \in V : f(x) = \ve0}$.
The nullity of $\phi$, denoted $\nullity(\phi)$, is the dimension of its kernel, that is, the dimension
of the smallest dimensional vector space containing $\ker(\phi)$.
The image of $\phi$
is defined as $\img(\phi) = \set{ f(x) : x \in V} \subseteq W$.
The rank of $\phi$, denoted $\rank(\phi)$, is, analogously to the nullity,
the dimension of the smallest dimensional vector space containing $\img(\phi)$.
If $\phi$ is a linear map, the definitions of rank and nullity simplify to
\[
	\nullity(\phi) = \dim(\ker(\phi)), \qquad \rank(\phi) = \dim(\img(\phi)) .
\]
The definitions are also used for matrices as a matrix $A \in \R^{m \times n}$ represents the linear map $\phi \colon \R^n \to \R^m,\, \vex \mapsto A\vex$.
The kernel of a matrix $A \in \R^{m \times n}$ is $\ker(A)=\set{ \vex \in \R^n : A\vex = \ve0}$, the nullity is its dimension,
the image is $\set{ A\vex \in \R^m : \vex \in \R^n }$ and the rank is its dimension. The latter coincides with the alternative rank
definition of matrices as the largest number $k$ such that any $k$ columns of $A$ are linearly
independent.

We make use of the well-known rank-nullity theorem, stating that a linear map splits the
dimension of the domain into the dimension of the kernel and the dimension of the image (cf. the splitting lemma).
\begin{theorem}[Rank-Nullity Theorem] \label{stmt:rank_nullity}
	Let $V,W$ be vector spaces, where $\dim(V) < \infty$, and $\phi \colon V \to W$ be a linear map. Then
	\[ \dim(V) = \nullity(\phi) + \rank(\phi) . \]
\end{theorem}
We are generally interested in the number of variables in an LP that attain non-integral
values. Thus, analogously to the support of a vector, we define the non-integral support.
\begin{definition}[Non-Integral Support]
\[ \nisupp(\vex) \coloneqq \set { j \in \supp(\vex) : x_j \notin \Z }. \]
\end{definition}

We will need the following proposition, which generalizes the fact that for a system $A\vex = \veb$, $\vex \geq 0$, the support of a vertex solution consists of linearly independent columns, to systems with (additional) lower and upper bounds on the variables:
\begin{proposition}[{Proposition 3.3.3, \cite{bertsekas2003convex}}] \label{stmt:bfslu}
	Let $\vex^*$ be a vertex solution to an LP of the form $A\vex = \veb$, $\vel \leq \vex \leq \veu$, $\vel,\veu \in \Q^n$.
	Then $(A_j)_{j \in \supp(\vex^*) : l_j < x_j < u_j}$ is non-singular.
\end{proposition}

Given that a vertex solution can be computed in polynomial time, the following lemma is straightforward via Proposition~\ref{stmt:bfslu}.

\begin{lemma} \label{stmt:nisupprank_general}
	Consider an LP with the objective of minimizing or maximizing $\vew\vex$ subject to $\set{\vex: A\vex=\veb, \, \vel\le \vex\le \veu,\, \vel,\veu\in \Z^{n} }$. In polynomial time we can compute an optimal solution $\vex^*$ such that 
    $(A_j)_{j \in \supp(\vex^*) : l_j < x_j < u_j}$ is non-singular.
\end{lemma}

\section{Approximation algorithms for IPs with a constant number of constraints - Proof of Theorem~\ref{conj:IP-general}}\label{sec:approx}

This section is devoted to the proof of Theorem~\ref{conj:IP-general}. We follow the general method we described in \lq\lq Technical overview\rq\rq. The whole proof is divided into three steps.

 \paragraph*{Step 1: Transform IP~\eqref{ILP:1} to a mixed IP with $\OO(1)$ integer variables.}

Recall that $H=(\veh_1,\veh_2,\dots,\veh_n)$ where $\veh_j\in \interval {-\Delta} {\Delta}^m$. We cut $\interval {-\Delta} {\Delta}^m$ into $(2/\delta)^m$ small boxes (called $\delta$-boxes), $\prod_{i=1}^m \interval {(\lambda_i-1)\delta\Delta} {\lambda_i\delta\Delta}$, where $\delta>0$ is an arbitrary small constant and $\lambda_i\in\{-1/\delta+1,-1/\delta+2,\dots,1/\delta\}$.  We arbitrarily index all the $\delta$-boxes, and let $I_k$ be the set of indices $j$'s such that $\veh_j$ is in the $k$-th $\delta$-box, where $k=1,\ldots, (2/\delta)^m$. Note that if some $\veh_j$ belongs to multiple $\delta$-boxes, let it be in arbitrary one such $\delta$-box.
 For the $k$-th $\delta$-box, we pick one fixed vector $\vev_k$ and call it the {\it canonical vector}. For each $j\in I_k$, we define $\tilde{\veh}_j=\veh_j-\vev_k$, so it is clear that $\|\tilde{\veh}_j\|_{\infty}\le \delta\Delta$.  
 
 We say that the vectors in the same $\delta$-box are similar. By introducing a new variable $y_k$ to the $k$-th $\delta$-box, where $y_k\coloneqq\sum_{j\in I_k}x_j$, we observe that $H\vex=\sum_{j=1}^n\veh_jx_j$ can be rewritten as $\sum_{k=1}^{(2/\delta)^m}\sum_{j\in I_k}(\vev_k+\tilde{\veh}_j)x_j$, thus obtaining the following MIP:
\begin{subequations}
	\begin{eqnarray}
	(\text{MIP}_1)  \quad  \min &&  \sum_{j=1}^nw_jx_j \nonumber\\
	&& \sum_{k=1}^{(2/\delta)^m}\vev_ky_k+\sum_{k=1}^{(2/\delta)^m}\sum_{j\in I_k} \tilde{\veh}_jx_j=\veb \label{mip:1}\\
	&&\sum_{j\in I_k}x_j=y_k, \quad \forall 1\le k\le (2/\delta)^m \label{mip:2}\\
	&&l_j\le x_j\le u_j,\quad \forall j\in I_k, 1\le k\le (2/\delta)^m \nonumber\\
	&&y_k\in \Z, x_j\in \R,	\quad \forall 1\le j\le n, 1\le k\le (2/\delta)^m \nonumber
	\end{eqnarray}
\end{subequations}

Notice that $(\text{MIP}_1)$ relaxes $x_j\in\Z$ to $x_j\in\R$, and only contains $(2/\delta)^m$ integer variables $y_k$. Hence, applying Kannan's algorithm~\cite{kannan1987minkowski}, an optimal solution to $(\text{MIP}_1)$ can be computed in $(2/\delta)^{\OO(m(2/\delta)^m)}\cdot |I|$ time. 

Let $\vey^*$ and $\vex^*$ be the optimal solution to $(\text{MIP}_1)$. In the following step, we will round fractional variables such that most fractional variables take integral values, and all constraints are still satisfied. 


\paragraph*{Step 2: Obtain a feasible solution with $\OO(m)$ variables taking a fractional value.}


Towards rounding $x_j^*$'s without changing the constraints in $(\text{MIP}_1)$, it suffices to consider the following LP.
\begin{subequations}
		\begin{eqnarray}
	 		(\text{LP}_2) \quad  \min &&  \sum_{j=1}^nw_jx_j \nonumber \\
	 		&& \sum_{k=1}^{(2/\delta)^m}\sum_{j\in  I_k}\tilde{\veh}_jx_j=\sum_{k=1}^{(2/\delta)^m}\sum_{j\in  I_k}\tilde{\veh}_jx_j^* \label{LP_r:1}\\
	 		&&\sum_{j\in  I_k}x_j=\sum_{j\in  I_k}x_j^*, \quad \forall 1\le k\le (2/\delta)^m \label{LP_r:2}\\
	 		&&l_j\le x_j\le u_j,  x_j\in \R,\quad \forall j\in  [n] \nonumber
	 	\end{eqnarray}
\end{subequations}

Let $\bar{\vex}=(\bar{x}_j)_{j=1}^n$ be an optimal solution to the above $(\text{LP}_2)$ computed via \autoref{stmt:nisupprank_general}.
We will argue on the number of variables in $\bar{\vex}$ taking a fractional value. Towards this, let $\eta \coloneqq |\nisupp(\bar{\vex})|$.

It is straightforward to show that $\eta\le \OO(\delta^{-m})$. However, this is not sufficient to bound the rounding error, since in this case the error can reach $\OO(\delta^{-m})\cdot \delta \Delta=\OO(\delta^{-m+1}\Delta)$, where $\delta>0$ is an arbitrary small constant. The following claim gives a sharper bound.

\begin{claim}\label{claim:1}
$\eta\le 2m$.
\end{claim}
\begin{claimproof} 
For simplicity, let the constraints of $(\text{LP}_2)$ be $A\vex=\veb$. Let $B = (A_j)_{j \in \nisupp(\bar{\vex})}$ be the submatrix induced by $\nisupp(\bar{\vex}$). Note that by \autoref{stmt:nisupprank_general}, $B$ is non-singular.


Observe that for every $k$ we either have $|\nisupp(\bar{\vex}) \cap I_k| = 0$ or $|\nisupp(\bar{\vex}) \cap I_k| \geq 2$ due to Constraint \eqref{LP_r:2}.
If $|\nisupp(\bar{\vex}) \cap I_k| = 0$, Constraint \eqref{LP_r:2} becomes a zero-row in $B$.
Let $C$ be the matrix $B$ after removing all zero-rows.
Then $C$ is non-singular (as we only removed zero rows from $B$). Moreover, since $I_k$'s are disjoint, there are at most $\eta/2$ constraints in Constraint \eqref{LP_r:2} which do not become a zero-row. Hence, there are at most $m + \frac \eta 2$ constraints and $\eta$ columns in $C$.
If $C$ has more columns than rows, then $C$ is singular, which is a contradiction.
So we have $m + \frac \eta 2 \geq \eta$, or equivalently $\eta \leq 2m$.
\end{claimproof}

Finally, we round the $\OO(m)$ variables in $\bar{\vex}$ that take fractional values.

\paragraph*{Step 3: Round the $\OO(m)$ fractional variables.}

Let $I_k^*\subseteq I_k$, $k=1,\dots,(2/\delta)^m$, be the indices of variables that still take a fractional value after Step 2. Without loss of generality, we may assume that for all $j$'s in a fixed $ I_k^*$, $w_j$'s are in non-decreasing order.  

Let $\gamma_k\coloneqq\sum_{j\in I_k^*}(\bar{x}_j-\lfloor \bar{x}_j\rfloor)=\sum_{j\in I_k}(\bar{x}_j-\lfloor \bar{x}_j\rfloor)$. Constraint~\eqref{LP_r:2} guarantees that $\gamma_k$ is an integer. 
Therefore, we can round up the first $\gamma_k$ fractional variables in $I_k^*$ (i.e., variables of the smallest indices) to $\lceil \bar{x}_j\rceil$, and meanwhile round down the other fractional variables in $I_k^*$ to $\lfloor \bar{x}_j\rfloor$. 
This rounding method preserves the constraint $\sum_{j\in I_k}x_j=y_k$, and meanwhile does not increase the objective value. However, this rounding will introduce $\OO(m\delta\Delta)$ error to Constraint~(\ref{mip:1}). Thus, Theorem~\ref{conj:IP-general} follows by taking $\delta=\OO(\frac{\varepsilon}{m})$.

\begin{description}
    \item[Running time.] 
\end{description}

Step 1: $(2/\delta)^{\OO(m(2/\delta)^m)}\cdot |I|=2^{(\frac{m}{\varepsilon})^{\OO(m)}}\cdot |I|$

Step 2: $n^4 \cdot |I|$

Step 3: $(2/\delta)^m\cdot n\log n$

All in all, the total time is $2^{(\frac{m}{\varepsilon})^{\OO(m)}}\cdot \textnormal{poly}(|I|)$, which has the form of $f(m,\varepsilon)\cdot\textnormal{poly}(|I|)$.



\section{Proof of Theorem~\ref{conj:nfold-additive}}
To prove Theorem~\ref{conj-coro:nfold-additive}, we first need to prove Theorem~\ref{conj:nfold-additive}.
We restate the theorem below.
\begin{th2}
	Given is an integer program $\min\{\vew\vex: \sum_{i=1}^nD^i\vex^i=\veb^0, \vex^i\in\mathcal{P}^i, 1\le i\le n,\, \vex\in \Z^{nt} \}$ with optimal objective value $\OPT$, where $D^i\in\Q^{s\times t}$ and $\mathcal{P}^i$ is an arbitrary set of integer vectors.  Let  $\kappa=\max_{\veu\in \cup_{i=1}^n\mathcal{P}^i}\|\veu\|_\infty$, i.e., $\kappa$ is the largest $\ell_\infty$-norm among all integer vectors in $\cup_{i=1}^n\mathcal{P}^i$. 
	Then for arbitrary small  $\varepsilon>0$, there exists an algorithm of running time ${f(s,t,\kappa,\varepsilon)}\cdot\textnormal{poly}(|I|)$ which returns a near-feasible solution $\tilde{\vex}$ such that $\tilde{\vex}\in\set{\vex: \|\sum_{i=1}^nD^i\vex^i-\veb^0\|_{\infty}\le \varepsilon\Delta, \,\vex^i\in\mathcal{P}^i,1\le i\le n,\, \vex\in \Z^{nt} }$, and $\vew\tilde{\vex}\le \OPT$, where $\Delta=\max_{i\in [n]}\|D^i\|_{\infty}$. 
\end{th2}

To prove Theorem~\ref{conj:nfold-additive}, we follow the same general strategy as the proof of Theorem~\ref{conj:IP-general}. That is, we first reformulate the input IP into an MIP via some appropriate $\delta$-boxes, and then we solve the MIP, fix the values of integer variables and obtain an LP, and argue that the vertex solution of this LP only contains a small number of variables taking a fractional value. Finally we round those fractional variables. The main challenge lies in the argument for the vertex solution, which builds upon a specific block structure of the constraint matrix. To facilitate this argument, we need to further introduce some notions. 

\subsection{General mechanisms for LP that consists of two parts}

In this section, we develop some preliminaries tailored to our setting. 

We will encounter the case where the constraint matrix can be partitioned into two parts/subsystems: One with small rank and one where
for a given solution either the non-integral support is small or we can easily find
linearly independent kernel elements acting only on the non-integral support. Assume we have an LP
of the following form:
\[ A \vex = \veb_A, \qquad C \vex = \veb_C, \qquad \vex \geq \ve0  , \]
where $A$ is any matrix, possibly with large rank, but $C$ has only a few (say, $\iota$) rows.
Consider some fixed solution $\vex^*$ and we want to reduce $\nisupp(\vex^*)$
as much as possible. If $A$ has large rank, the combined system also has large rank,
so it is not really helpful trying to obtain a solution with a small non-integral support through arguing a small rank. However, assume that we are able
to find a set $\mathcal K$, $|\mathcal K| \geq \iota + 1$, of linearly independent non-trivial kernel elements of $A$, that are only non-zero on variables in $\nisupp(\vex)$. Then we can ``combine'' those
kernel elements to a non-trivial kernel element of the combined system, using \autoref{stmt:rank_nullity}.

More formally, the following lemma allows us to show non-zero nullity of the combined system
if enough kernel elements of the first system (more than the rank of the second one)
can be found.
\begin{lemma} \label{stmt:connect}
	Consider vector spaces $V,W$ where $\dim(V) < \infty$ and two linear maps
	$\phi_1,\phi_2 \colon V \to W$. If $\nullity(\phi_1) > \rank(\phi_2)$, then
	$\dim(\ker(\phi_1) \cap \ker(\phi_2)) > 0$.
\end{lemma}
\begin{proof}
	Let $\psi \colon \ker(\phi_1) \to W,\, x \mapsto \phi_2(x)$ be the restriction of $\phi_2$ to $\ker(\phi_1)$.
	Then $\rank(\psi) \leq \rank(\phi_2)$ and $\ker(\phi_1) \cap \ker(\phi_2) = \ker(\psi)$, and thus
	\[
		\dim(\ker(\phi_1) \cap \ker(\phi_2))
		= \nullity(\psi)
		\overann={\text{\autoref{stmt:rank_nullity}}} \dim(\ker(\phi_1)) - \rank(\psi)
		\overann\geq{\rank(\psi) \leq \rank(\phi_2)} \nullity(\phi_1) - \rank(\phi_2)
		\overann>{\text{assumption}} 0  . \qedhere
	\]
\end{proof}
The scenario mentioned above is then covered by applying this lemma with $V = \ker(A)$, $W = \img(C)$,
$\phi_1 \colon \vex \mapsto A \vex$, $\phi_2 \colon \vex \mapsto C \vex$: Given that $\nullity(A) > \rank(C)$,
there is a non-trivial kernel element of the combined system. Note that $\nullity(A) > \rank(C)$ translates
to ``there are at least $\rank(C) + 1$ linearly independent non-trivial kernel elements of $A$''.

If the constraint matrix of an LP is totally unimodular and the right-hand side is integral, it is well-known that vertex solutions are integers. Based on this, we can easily obtain the following proposition.  
\begin{proposition} \label{stmt:tu_aug}
	Let $\vex$ be a solution to an LP of the form $A\vex = \veb$, $\vel \leq \vex \leq \veu$, $\vel,\veu \in \Q^m$, $\veb \in \Z^m$, where the constraint matrix $A$ is totally unimodular. 
	Then if $\nisupp(\vex) \neq \emptyset$, we have $\nullity((A_j)_{j \in \nisupp(\vex)}) > 0$.
\end{proposition}
\begin{proof}
	Let $B$ be the matrix $A$ restricted to the columns from $\nisupp(\vex)$. Since $A$ is totally unimodular, so is $B$. 
	Let $\vey$ be the vector $\vex$ restricted to entries in $\nisupp(\vex)$.
	Then $B\vey$ is integral, as $Ax$ is integral and only $x_j$'s that are integral have been deleted from $\vex$ to obtain $\vey$.
	Let $\gamma \coloneqq A\vex - B\vey$.
	Then $\gamma$ is an integer vector as $A\vex$ and $B\vey$ also are. Equivalently, $B\vey = A\vex - \gamma$.
	Consider the LP $B\vez = A\vex - \gamma$ with $\vez$ being variables. It is feasible (due to $\vey$). Further, there is an integer vector $\vey^*$ such that $B\vey^* = A\vex - \gamma$, since $B$ is totally unimodular.
	By construction, $\vey$ has no integral components, so $\vey - \vey^*$ is a non-trivial kernel element of $B$, directly implying $B$, and hence also $A$, has non-zero nullity.
\end{proof}
Rephrasing the above: A solution to a totally unimodular system with integral right-hand side is either
integral or the subsystem induced by its non-integral variables is underdetermined.

\subsection{Proof of Theorem~\ref{conj:nfold-additive}}\label{sec:3.2}
Now we are ready to prove Theorem~\ref{conj:nfold-additive}.
 Let $\tau\coloneqq\max_{i\in [n]}|\mathcal{P}^i|=\kappa^t$. By adding dummy vectors, we may assume $|\mathcal{P}^i|=\tau$ for all $i$ and let $\mathcal{P}^i=\{\vep^i_1,\vep^i_2,\dots,\vep^i_{\tau}\}$. Consider the matrix $\mathcal{D}^i=(D^i\vep^i_1,D^i\vep^i_2,\dots,D^i\vep^i_{\tau})$ and notice that $\mathcal{D}^i$ is fixed. Let $\Delta=\max_{i\in [n]}\|D^i\|_\infty$. It is clear that $\|\mathcal{D}^i\|_\infty=\Delta\kappa t$. 
 
To prove Theorem~\ref{conj:nfold-additive}, we shall follow the general strategy for proving Theorem~\ref{conj:IP-general}. We first cut the box $[-\|\mathcal{D}^i\|_\infty,\|\mathcal{D}^i\|_\infty]^s$ into a constant number of small boxes, introduce one hyper-integer variable for each small box, and obtain an MIP, which can be solved in polynomial time. Finally, we round fractional variables to integer variables.

Towards the proof, we first rewrite the input IP in a form that is easier for our analysis.

\paragraph*{Step 0: Rewrite the input IP.}
We cut the box $\interval{-\|\mathcal{D}^i\|_\infty}{\|\mathcal{D}^i\|_\infty}^s$ into $(\frac{2}{\delta})^{s}$ small $\delta$-boxes of the form $\prod_{i=1}^s \interval{(\lambda_i-1)\delta\Delta\kappa t}{\lambda_i\delta\Delta\kappa t}$, where $\lambda_i\in\{-1/\delta+1,-1/\delta+2,\dots,1/\delta\}$, and $\delta$ is a sufficiently small constant. In particular, it will become clear later that taking $\delta=\OO(\frac{\varepsilon}{st\kappa^{t+1}})$ suffices.

 We pick a fixed vector within each $\delta$-box and let it be the canonical vector of this $\delta$-box. Each $D^i\vep^i_j$ lies in some $\delta$-box and corresponds to the canonical vector of this $\delta$-box. Hence, $\mathcal{D}^i=\mathcal{V}^i+\tilde{\mathcal{D}}^i$ where $\mathcal{V}^i$ is the matrix consisting of the canonical vectors corresponding to the column vectors of $\mathcal{D}^i$, and $\tilde{\mathcal{D}}^i$ is the ``residue matrix'' satisfying that $\|\tilde{\mathcal{D}}^i\|_\infty\le \delta\Delta\kappa t$. 
 
Since each $\mathcal{V}^i$ has $\tau$ columns, it is easy to see that there are $\rho\coloneqq((\frac{2}{\delta})^{s})^\tau=(\frac{2}{\delta})^{(s\kappa^t)}$ different types of canonical matrices $\mathcal{V}^i$'s. For $k=1,\ldots,\rho$, let $I_k$  denote the set of indices $i$'s (of $\mathcal{D}^i$'s) whose canonical matrix is of type $k$. The constraint $\sum_{i=1}^nD^i\vex^i= \veb^0$ essentially asks for selecting some column from each $\mathcal{D}^i$ such that they add up to $\veb^0$, so we rewrite it with the help of auxiliary variables $\vez$ which denote such a choice: Let $z_{ij}=1$ represent that we select the $j$-th column of $\mathcal{D}^i$ and $z_{ij}=0$ otherwise.  We have the following equivalent IP:
 \begin{subequations}
 	\begin{eqnarray}
 	(\text{IP}_3)  \quad  \min && \sum_{i=1}^{n}(\sum_{j=1}^{\tau} \vew^i\vep^i_jz_{ij}) \label{mip:02}\\
 	&& \sum_{k=1}^{\rho}\sum_{i\in I_k}\sum_{j=1}^\tau (D^i\vep^i_j)z_{ij} = \veb^0 \label{mip:12}\\
 	&&\sum_{j=1}^{\tau}z_{ij}=1, \quad \forall 1\le i\le n \label{mip:22}\\
 	&&z_{ij}\in \{0,1\},	\quad \forall 1\le i\le n, 1\le j\le \tau \label{mip:32}
 	\end{eqnarray}
 \end{subequations}

\paragraph*{Step 1: Transform $(\text{IP}_3)$ to a mixed IP with $\OO(1)$ integer variables.}
 
Recall that there are $\rho$ distinct types of canonical matrices $\mathcal{V}^i$'s. We let $(\vev^k_1,\vev^k_2,\dots,\vev^k_{\tau})$ be the $k$-th type. Suppose that the canonical matrix of $\mathcal{D}^i$ is of type $k$, and 
 then $\tilde{\mathcal{D}}^i=(D^i\vep^i_1-\vev^k_1,D^i\vep^i_2-\vev^k_2,\dots,D^i\vep^i_{\tau}-\vev^k_{\tau})$. 

Similar to the proof of Theorem~\ref{conj:IP-general}, we introduce $y_{kj}=\sum_{i\in I_k}z_{ij}$ as the group variable that counts how many $j$-th columns we have selected among all $\mathcal{D}^i$'s whose canonical matrix is of type $k$. There are $\rho\tau=(\frac{2}{\delta})^{(s\kappa^t)}\kappa^t$ group variables. We establish $(\text{MIP}_4)$ below, where $y_{kj}$'s are integer variables and $z_{ij}$'s are fractional variables.
\begin{subequations}
 	\begin{eqnarray}
 	(\text{MIP}_4)  \quad  \min && \sum_{i=1}^{n}(\sum_{j=1}^{\tau} \vew^i\vep^i_jz_{ij})  \label{mip:03}\\
 	&& \sum_{k=1}^{\rho}\sum_{j=1}^\tau \vev_j^k y_{kj}+ \sum_{k=1}^{\rho}\sum_{i\in I_k}\sum_{j=1}^\tau (D^i\vep^i_j-\vev_j^k)z_{ij} = \veb^0  \label{mip:13}\\
 		&& y_{kj}=\sum_{i\in I_k}z_{ij}, \quad \forall 1\le j\le \tau, 1\le k\le \rho\label{mip:43}\\
 	&&\sum_{j=1}^{\tau}z_{ij}=1, \quad \forall  i\in I_k, 1\le k\le \rho \label{mip:23}\\
 &&z_{ij}\in[0,1], y_{kj}\in\Z, 	\quad \forall  1\le j\le \tau, i\in I_k, 1\le k\le \rho \label{mip:33}
 	\end{eqnarray}
 \end{subequations}
 
 Notice that in $(\text{MIP}_4)$, the coefficients $\vew^i\vep^i_j$, $\vev_j^k$, and $(D^i\vep^i_j-\vev_j^k)$ are fixed values instead of variables. 
 Observe that $(\text{MIP}_4)$ only contains $\rho\tau$ integer variables $y_{kj}$. Thus, applying Kannan's algorithm~\cite{kannan1987minkowski}, an optimal solution to $(\text{MIP}_4)$ can be computed in $(\rho\tau)^{\OO(\rho\tau)}\cdot |I|$ time, which is polynomial.
 
 Let $y_{kj}=y_{kj}^*\in\Z$ and $z_{ij}=z_{ij}^*\in [0,1]$ be the optimal solution to $(\text{MIP}_4)$. 

\paragraph*{Step 2:  Obtain a feasible solution with $\OO(s\tau)$ variables taking a fractional value.}

Fix $\vey=\vey^*$, and then $(\text{MIP}_4)$ becomes an LP. We see that 
this LP is a combined system: We denote Constraint \eqref{mip:13} as $C\vez=\veb_C$, and Constraints \eqref{mip:43} and \eqref{mip:23} as $A_\Sigma\vez=\veb_A$. We denote by $\mathrm{LP}_\Sigma$ the linear system obtained by removing $C\vez=\veb_C$. That is, the constraint matrix of $\mathrm{LP}_\Sigma$ is exactly $A_\Sigma$.

Observe that $A_\Sigma = \begin{pmatrix} A_1 && \\ & \ddots & \\ && A_\rho \end{pmatrix}$ is a block
diagonal matrix, as for each $k \in \iv \rho$, the set of variables $I_k \times \iv \tau$ appearing in the constraints is disjunct with those for any other $k' \in \iv \rho \setminus \{ k \}$.
Moreover, each $A_k\in \{0,1\}^{(|I_k|+\tau)\times (|I_k|\cdot\tau)}$ is the constraint matrix of an assignment problem and thus
totally unimodular (and thus so is $A_\Sigma$). More precisely, for $k=1,\dots,\rho$, 
\[
     \begin{array}{c@{\hspace{-5pt}}l}   
     A_k=\left(
     \begin{array}{ccc:ccc:c:ccc}
       1   &  &  & 1 &  &  & &1&&\\
       & \ddots &  &  & \ddots&&\cdots &&\ddots& \\
       & & 1 & &&1&&&&1 \\ 
       \hdashline
      1 & \cdots & 1&  &  &  && & & \\
        &  & & 1 & \cdots & 1 && & & \\
         & & &  & &&\ddots &  & &  \\
          &  & &  &  &  && 1&\cdots &1 \\
     \end{array}
     \right)
     &
     \begin{array}{l}
          \left.\rule{0mm}{7.6mm}\right\}|I_k| \text{ rows} \\  
          \\
          \left.\rule{0mm}{10.2mm}\right\}\tau \text{ rows}
     \end{array}
     \\[-5pt]
     \begin{array}{ccr}
       \hspace{23pt}  \underbrace{\rule{14mm}{0mm}}_{|I_k| \text{ columns}} &  
          \underbrace{\rule{14mm}{0mm}}_{|I_k| \text{ columns}}& {\hspace{20pt}
           \underbrace{\rule{14mm}{0mm}}_{|I_k| \text{ columns}}}
     \end{array}
     &  
     \end{array}
     \]
where the first $|I_k|$ rows are for Constraint~\eqref{mip:23}, and the last $\tau$ rows are for Constraint~\eqref{mip:43}.

\begin{remark*}
It is clear that Constraint \eqref{mip:13} together with $A_\Sigma\vez=\veb_A$ form an $n$-fold structure. However, the number of rows and columns of each $A_k$ can be prohibitively large. We mentioned in Section~\ref{subsec:compare} before that Cslovjecsek et al.~\cite{cslovjecsek2021block} established a general result which implies that in a vertex solution satisfying $C\vez=\sum_kC_k\vez^k=\veb_C$ and $A_\Sigma\vez=\veb_A$, at most $s$ vectors out of $\vez^1,\vez^2,\dots,\vez^{\rho}$ may take fractional values, but each $\vez^k$ contains $|I_k|\cdot \tau$ fractional variables, which means the total number of fractional variables is still too large. Our goal is to upper bound the number of fractional variables to $\OO(s\tau)$, which requires a more fine-grained analysis that leverages the special structure of $A_k$'s.
\end{remark*}

Interestingly, for $A_k$ it is easy to find a comparably small kernel element if 
solutions have enough non-integral components (for integral right-hand sides):
\begin{lemma} \label{stmt:rank}
	For $k \in \iv \rho$, consider a vertex solution $\vez$ to $A_k \vez = \veb$ for 
	some integer vector $\veb$. Let $B$ be the matrix $A_k$ restricted to the columns in $\nisupp(\vez)$.
	Then $\rank(B) \leq 2\tau$. 
\end{lemma}
\begin{proof}
	If $|\nisupp(\vez)| \leq 2\tau$, the statement is trivial. Thus, assume $|\nisupp(\vez)| \geq 2\tau + 1$.
	For the sake of contradiction, assume $\rank (B) >  2\tau$. Let $I \subseteq \nisupp(\vez)$ be a set of $|I| = 2 \tau + 1$ column indices. Let $B'$ be the matrix $B$ restricted to the columns in $I$, where afterwards
	all zero-rows have been removed. Clearly, $2 \tau + 1 = \rank(B') \leq \rank(B)$ as we only removed zero-rows.
	Observe that every row in $\rank( B' )$ that corresponds to Constraint~\eqref{mip:23}
	has to contain at least two non-zero variables due to the integral right-hand side. Moreover,
	all the constraints in \eqref{mip:23} are variable-disjunct. So the number of rows in
	$B'$ is at most 
	\[ \tau + \frac {|I|} 2 = \tau + \frac{2\tau + 1} 2 < 2 \tau + 1 .\]

	However, we have $|I| = 2\tau + 1$ columns, so $B'$ has non-zero nullity, which,
	by \autoref{stmt:rank_nullity}, is a contradiction
	to $\rank(B') = 2 \tau + 1$.
\end{proof}
\begin{lemma} \label{stmt:nullity_sigma}
	Let $\hat \vez$ be a solution to $\mathrm{LP}_\Sigma$, and $B_\Sigma$ the matrix $A_\Sigma$ restricted to
	the columns $\nisupp(\hat \vez)$. Then we have 
	\[ \nullity(B_\Sigma) \geq \frac{ |\nisupp(\hat \vez)| }{ 2\tau + 1 } . \]
\end{lemma}
\begin{proof}
	Observe that since $A_\Sigma$ is a block diagonal matrix, to prove this lemma, it suffices to show for every $k \in \iv \rho$ that $B_k$, the restriction of $A_k$ to $\nisupp(\hat \vez) \cap I_k$, we have
	(where $\eta \coloneqq |\nisupp(\hat \vez) \cap  I_k|$):
	\[ 
	\nullity(B_k) \geq {\frac{\eta}{2 \tau + 1}} 
	.\]
	If $\eta = 0$, the statement is trivial.
	If $0 < \eta \leq 2 \tau + 1$, we only have to show that the nullity is non-zero, which follows
	directly from \autoref{stmt:tu_aug}. For $2 \tau + 1 < \eta$, by \autoref{stmt:rank} and rank-nullity theorem, the nullity
	is at least $\eta - 2 \tau$. Observe that for any $\bar{\iota}$ we have
	\[ \eta - \bar{\iota} \geq \frac \eta {2 \tau +1} 
	\iff \eta - \frac \eta {2 \tau + 1} \geq\bar{\iota}
	\iff \frac{(2 \tau + 1)\eta - \eta}{2 \tau + 1} \geq \bar{\iota}
	\iff \frac{2 \tau \eta}{2 \tau + 1} \geq \bar{\iota} .
	\]
	Furthermore, as $\eta > 2 \tau + 1$, then $\frac{2 \tau \eta}{2\tau + 1} \geq \frac{2 \tau (2 \tau + 2)}{2 \tau + 1} > 2 \tau,$ i.e.,  $\eta - \bar{\iota} \geq \frac \eta {2 \tau +1}$ holds in particular for $\bar{\iota} \coloneqq 2 \tau$,
	which eventually leads to
	\[ 
		\frac{\eta}{2 \tau + 1}
		\leq \eta - 2 \tau \leq \nullity(B_k) . \qedhere
	\]
\end{proof}
Now, let us consider vertex solution $\gamma$ to the combined linear system, 
computed by \autoref{stmt:nisupprank_general}.
Let $B_\Sigma$ be the restriction of $A_\Sigma$ to the columns in $\nisupp(\gamma)$.
We claim that $s \geq \frac{ |\nisupp(\gamma)| }{ 2 \tau + 1 }$.
Assume otherwise.
By \autoref{stmt:nullity_sigma}, $\frac{ |\nisupp(\gamma)| }{ 2 \tau + 1 } \leq \nullity(B_\Sigma)$.
By \autoref{stmt:connect}, having $\nullity(B_\Sigma) > s$ suffices to prove the existence of 
a non-trivial kernel element on $\nisupp(\gamma)$ of the combined system, which is the one that includes Constraint~\eqref{mip:13}.
This is a contradiction to \autoref{stmt:nisupprank_general} which postulates that the constraint matrix restricted to $\nisupp(\gamma)$ is non-singular. 

Therefore, it follows that $ |\nisupp(\gamma)| \le s(2\tau+1)=\OO(s\tau)$.

\paragraph*{Step 3: Round the $\OO(s\tau)$ fractional variables.} 

We can round the fractional variables for each $k$, respectively, since the constraints for different $k$'s cannot affect each other. For a fixed $k$, by Constraint~(\ref{mip:43}), we have $y_{kj}=\sum_{i\in I_k}z_{ij}, \forall 1\le j\le \tau$, and by Constraint~(\ref{mip:23}), it holds that $\sum_{j=1}^{\tau}z_{ij}=1,  \forall  i\in I_k$. We assume $I_k=\{i_1,i_2,\ldots,i_\ell\}$, and then Constraints~(\ref{mip:43}) and~(\ref{mip:23}) are shown as Figure~\ref{fig:2}.
 
 \begin{figure}[ht]
 	\centering
\includegraphics[width=1\linewidth]{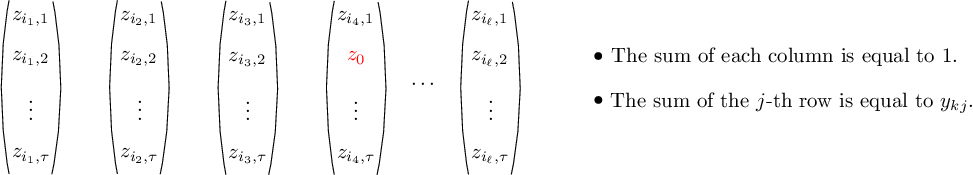}
 	\caption{Constraints~(\ref{mip:43}) and~(\ref{mip:23}) for a fixed $k$.}
 	\label{fig:2}
 \end{figure}
 
 Check all fractional variables $z_{ij}$'s, and pick up the fractional variable with the smallest coefficient in the objective function~(\ref{mip:03}). Without loss of generality, we let $z_0$ be this fractional variable. See Figure~\ref{fig:2} as an illustration. Then we round up the value of $z_0$ to 1. By the constraint $\sum_{j=1}^{\tau}z_{ij}=1,  \forall  i\in I_k$, we round down the values of these variables in the same column as $z_0$ to zeros. This preserves the constraint $\sum_{j=1}^{\tau}z_{ij}=1$. 
 
 Next, for the fractional variables in the same row as $z_0$, we apply a greedy rounding similar as that in Step 3 of Theorem~\ref{conj:IP-general}. We assume $j_0$ is the row number that $z_0$ has. Since $y_{kj}=\sum_{i\in I_k}z_{ij}, \forall 1\le j\le \tau$, without loss of generality we assume that $z_{ij_0}$'s where $i\in I_k$ in the summation are ordered by the non-decreasing coefficients in the objective function~(\ref{mip:03}). We compute $\sum_{i\in I_k}(z_{ij_0}-\lfloor z_{ij_0}\rfloor)$. Let $\gamma_{kj_0}\coloneqq\sum_{i\in I_k}(z_{ij_0}-\lfloor z_{ij_0}\rfloor)$. Then we pick out the first $\gamma_{kj_0}$ variables, and round up them to 1 and round down the remaining fractional variables $z_{ij_0}$'s to zeros. 
This rounding method will preserve the constraint $y_{kj_0}=\sum_{i\in I_k}z_{ij_0}$. 

After the previous two processes, we delete the column and row where $z_0$ is located. Iteratively find the fractional variable with the smallest coefficient in the objective function from the current fractional variables, and repeat the above operations. Clearly, the constraints $\sum_{j=1}^{\tau}z_{ij}=1,  \forall  i\in I_k$ and $y_{kj}=\sum_{i\in I_k}z_{ij}, \forall 1\le j\le \tau$ are always satisfied.

We round up or down all fractional variables by preserving Constraints~(\ref{mip:43}) and~(\ref{mip:23}) and meantime not increasing the objective, while this introduces $\OO(s\tau\delta\Delta\kappa t)$ error to Constraint~(\ref{mip:13}). Thus, Theorem~\ref{conj:nfold-additive} follows by taking $\delta=\OO(\frac{\varepsilon}{s\tau\kappa t})=\OO(\frac{\varepsilon}{st\kappa^{t+1}})$. 

\begin{description}
    \item[Running time.] 
\end{description}

Step 1: $(\rho\tau)^{\OO(\rho\tau)}\cdot |I|=((\frac{2}{\delta})^{(s\kappa^t)}\kappa^t)^{\OO((\frac{2}{\delta})^{(s\kappa^t)}\kappa^t)}\cdot |I|=2^{(\frac{s\kappa^t}{\varepsilon})^{\OO(s\kappa^t)}}\cdot |I|$

Step 2: $(\rho n)^4\cdot |I|=(\frac{2}{\delta})^{\OO(s\kappa^t)}n^4\cdot |I|$

Step 3: $\rho\tau^2 n^2\log n=((\frac{2}{\delta})^{s})^\tau \tau^2 n^2\log n=(\frac{2}{\delta})^{\OO(s\kappa^t)}n^2\log n$

All in all, the total time is $2^{(\frac{s\kappa^t}{\varepsilon})^{\OO(s\kappa^t)}}\cdot \textnormal{poly}(|I|)$, which has the form of $f(s,t,\kappa,\varepsilon)\cdot\textnormal{poly}(|I|)$.

\section{Approximation algorithms for $n$-fold IPs - Proof of Theorem~\ref{conj-coro:nfold-additive}}

This section is devoted to the proof of Theorem~\ref{conj-coro:nfold-additive}. We restate the theorem below. 


 
 \begin{th3}
Given is an integer program $\min\{\vew\vex: \sum_{i=1}^nD^i\vex^i=\veb^0, A^i\vex^i= \veb^i,1\le i\le n,\, \ve0\le \vex\le \veu, \vex\in \Z^{nt} \}$ with optimal objective value $\OPT$, where $A^i\in\Q^{s_A\times t_A}_{\ge0}$, $D^i\in\Q^{s_D\times t_D}_{\ge 0}$, and $t_A=t_D=t$. 
	Then for arbitrary small  $\varepsilon>0$, there exists an algorithm of running time ${f(s_A,s_D,t,\varepsilon)}\cdot\textnormal{poly}(|I|)$ which returns a near-feasible solution $\tilde{\vex}$ such that $\tilde{\vex}\in\{\vex: (1-\varepsilon)\veb^0\le\sum_{i=1}^nD^i\vex^i\le (1+\varepsilon)\veb^0,\, (1-\varepsilon)\veb^i\le A^i\vex^i\le (1+\varepsilon)\veb^i, 1\le i\le n,\, \ve0\le \vex\le\veu,\, \vex\in \Z^{nt} \}$, and $\vew\tilde{\vex}\le \OPT$. 
\end{th3}

The above theorem is closely related to Theorem~\ref{conj:nfold-additive}. In particular, since all $A^i$'s are positive matrices, if all entries in $A^i$ is large compared to $\veb^i$, then we know each coordinate of $\vex^i$ must be small. Consequently, $\vex^i$ may only take $\OO(1)$ possible values, in other words, $\vex^i\in \mathcal{P}^i$ where $|\mathcal{P}^i|$ is fixed, and Theorem~\ref{conj:nfold-additive} is applicable. What if some entries of $A^i$ are small, such that some coordinate of $\vex^i$, say, $x^i_j$, can take a large value? In this case, we observe that when the variable $x^i_j$ increases or decreases by a small amount, the equation $A^i\vex^i=\veb^i$ may only be violated by some small amount. This motivates us to split $\vex^i$ into two parts, the major part and the minor part, where the major part can only take $\OO(1)$ possible values, and the minor part can be handled like a fractional variable. 

More precisely, let $A^i_j$ be the $j$-th column of matrix $A^i$. By the constraint $A^i\vex^i= \veb^i$,  we have 
\begin{eqnarray*}
& A^i_1x^i_1+A^i_2x^i_2+\cdots+A^i_tx^i_t= \veb^i.
\end{eqnarray*}
We classify all $A^i_j$'s into two types: 
\begin{itemize}
\item If there is a component in vector $A^i_j$ with a value at least $\psi$, then the vector $A^i_j$ is defined as {\bf{big}}. 
\item Otherwise, if the values on all components of vector $A^i_j$ are less than $\psi$, then the vector $A^i_j$ is defined as {\bf{small}}. 
\end{itemize}
Here, $\psi$ is a parameter to be decided later. In fact, we will take $\psi\coloneqq\frac{\varepsilon}{2t}$. We divide the proof of Theorem~\ref{conj-coro:nfold-additive} into two cases. 

\noindent{\bf{Case 1.}} If all $A^i_j$'s are big vectors in $A^i\vex^i= \veb^i$ for all $i$'s, then we have $\sum_{j=1}^t x^i_j=\OO(s_A/\psi)$, which implies that $\|\vex^i\|_\infty\le\|\vex^i\|_1=\OO(s_A/\psi)$. By Theorem~\ref{conj:nfold-additive}, there exists an algorithm of running time ${f(s_A,s_D,t,\varepsilon)}\cdot\textnormal{poly}(|I|)$ which returns a near-feasible solution $\tilde{\vex}$ such that $\tilde{\vex}\in\{\vex: (1-\varepsilon)\veb^0\le\sum_{i=1}^nD^i\vex^i\le (1+\varepsilon)\veb^0, \,A^i\vex^i= \veb^i,\, \ve0\le \vex\le \veu, \vex\in \Z^{nt} \}$, and $\vew\tilde{\vex}\le \OPT$. Theorem~\ref{conj-coro:nfold-additive} follows.
\vspace{3pt}

 \noindent{\bf{Case 2.}} Otherwise, 
for every $A^i_j$ we can multiply it with some $\lambda_j^i$ to make it big. If some $A^i_j$ is initially big, let its $\lambda_j^i$ be 1. Then the constraint $A^i\vex^i= \veb^i$ can be written as
\begin{eqnarray*}
A^i_jx^i_j=\lambda^i_jA^i_j\left\lfloor\frac{x^i_j}{\lambda^i_j}\right\rfloor+A^i_j\left(x^i_j-\lambda^i_j\left\lfloor\frac{x^i_j}{\lambda^i_j}\right\rfloor\right),
\end{eqnarray*}
where $\lambda^i_j$ is the smallest integer such that $\lambda^i_jA^i_j$ is big. Denote $(x^i_j)'\coloneqq\lfloor\frac{x^i_j}{\lambda^i_j}\rfloor$, and $(x^i_j)''\coloneqq x^i_j-\lambda^i_j\lfloor\frac{x^i_j}{\lambda^i_j}\rfloor$. 
Then we have 
\begin{eqnarray*}
A^i_jx^i_j=\lambda^i_jA^i_j(x^i_j)'+A^i_j(x^i_j)''.
\end{eqnarray*}
 Note that the coefficients of $(x^i_j)'$ are big, while that of $(x^i_j)''$ are not. We can bound $\sum_{j=1}^t A^i_j(x^i_j)''$ by a negligible value due to a simple analysis. It implies that all possible values of $(x^i_j)'$ can be enumerated. 
 
 Similarly, the other constraint $\sum_{i=1}^nD^i\vex^i=\veb^0$ can be   written as 
 \begin{eqnarray*}
D^i_jx^i_j=\lambda^i_jD^i_j(x^i_j)'+D^i_j(x^i_j)''.
\end{eqnarray*}
Then we formulate a new IP taking $(x^i_j)'$ and $(x^i_j)''$ as variables. In the new IP, the constraint $A^i\vex^i=\veb^i$ can be safely replaced by $(\vex^i)'\in \mathcal{P}^i$.  We shall deal with the two summands $\sum_{i=1}^n\sum_{j=1}^t \lambda^i_jD^i_j(x^i_j)'$ and $\sum_{i=1}^n\sum_{j=1}^t D^i_j(x^i_j)''$ separately. The way we deal with the former is the same as Step 1 of Theorem~\ref{conj:nfold-additive}, and the way we deal with the latter is the same as Step 1 of Theorem~\ref{conj:IP-general}. Afterwards, we apply the rounding procedures specified in Theorem~\ref{conj:nfold-additive} and Theorem~\ref{conj:IP-general}.

We defer the detailed proof of Theorem~\ref{conj-coro:nfold-additive} to Appendix~\ref{appendix-1}.

\section{Conclusion}
We study approximation algorithms for IPs with a constant number of constraints and $n$-fold IPs. For IPs with a constant number of constraints, our algorithm returns a solution that violates the constraints by an additive error of $\OO(\varepsilon\Delta)$. For $n$-fold IPs, our algorithm returns a solution within a multiplicative error of $1+\varepsilon$. It is not clear whether similar results can be obtained for other block-structured IPs such as 2-stage stochastic IPs or $4$-block $n$-fold IPs. It is also an interesting open problem whether an additive approximation scheme can be obtained for $n$-fold IPs.



\bibliography{lipics-v2021-sample-article}

\clearpage
\appendix

\section{Proof of Theorem~\ref{conj-coro:nfold-additive}}~\label{appendix-1}

\begin{th3}
Given is an integer program $\min\{\vew\vex: \sum_{i=1}^nD^i\vex^i=\veb^0, A^i\vex^i= \veb^i,1\le i\le n,\, \ve0\le \vex\le \veu, \vex\in \Z^{nt} \}$ with optimal objective value $\OPT$, where $A^i\in\Q^{s_A\times t_A}_{\ge0}$, $D^i\in\Q^{s_D\times t_D}_{\ge 0}$, and $t_A=t_D=t$.   
	Then for arbitrary small $\varepsilon>0$, there exists an algorithm of running time ${f(s_A,s_D,t,\varepsilon)}\cdot\textnormal{poly}(|I|)$ which returns a near-feasible solution $\tilde{\vex}$ such that $\tilde{\vex}\in\{\vex: (1-\varepsilon)\veb^0\le\sum_{i=1}^nD^i\vex^i\le (1+\varepsilon)\veb^0,\, (1-\varepsilon)\veb^i\le A^i\vex^i\le (1+\varepsilon)\veb^i, 1\le i\le n,\, \ve0\le \vex\le\veu,\, \vex\in \Z^{nt} \}$, and $\vew\tilde{\vex}\le \OPT$. 
\end{th3}

\begin{proof}
Since $A^i\vex^i= \veb^i$, $\vex\ge \ve0$, and $A^i\in\Q^{s_A\times t_A}_{\ge0}$, we know $\veb^i\in\Q^{s_A}_{\ge0}$. For ease of discussion, we scale the values of all non-zero components of each vector $\veb^i$ to $1$, and keep the values on the zero components unchanged. By doing so, we have that $\veb^i\in \{0,1\}^{s_A}$ for all $i=1,\ldots,n$. 

Let $A^i_j$ be the $j$-th column of matrix $A^i$. By the constraint $A^i\vex^i= \veb^i$,  we have 
\begin{eqnarray}\label{column matrix}
& A^i_1x^i_1+A^i_2x^i_2+\cdots+A^i_tx^i_t= \veb^i.
\end{eqnarray}
If, say, some $h$-th coordinate of $\veb^i$ is $0$, we consider the $h$-th coordinate of $A_j^i$'s. There are two possibilities: (i). The $h$-th coordinate of some $A^i_j$ is nonzero. Then $x^i_j=0$ in any feasible solution. We simply fix $x_j^i=0$ and ignore this variable (as well as its corresponding $A_j^i$) throughout our following discussion. (ii). Assume without loss of generality that the $h$-th coordinate of every $A^i_j$ is $0$. Then the constraint induced by the $h$-th coordinates becomes $\sum_{j=1}^t 0\cdot x^i_j=0$, and we simply remove this constraint, i.e., we remove the $h$-th coordinate from $\veb^i$ and $A^i_j$. Hence, without loss of generality we may assume that $\veb^i=1$ for all $i$. From now on we still write $\veb^i$ for consistency, but notice that during all our subsequent discussion its zero entries are ignored. 



 \paragraph*{Step 0: Classify $A^i_j$'s as big or small, and establish an equivalent IP.} 

We fix $\psi\coloneqq\frac{\varepsilon}{2t}$ and classify all $A^i_j$'s into two different groups: If there is a component in vector $A^i_j$ with a value at least $\psi$, then the vector $A^i_j$ is defined as {\bf{big}}. Otherwise, if the values on all components of vector $A^i_j$ are less than $\psi$, then the vector $A^i_j$ is defined as {\bf{small}}.  According to~Eq(\ref{column matrix}), there are at most $\OO(1/\psi)$ big $A^i_j$'s in every dimension. Since $\veb^i=1$, there are at most $\OO(s_A/\psi)$ big $A^i_j$'s for each fixed $i$, that is, there are at most $\OO(s_A/\psi)$ non-zero $x^i_j$'s whose corresponding coefficients are big vectors for each fixed $i$.

If all $A^i_j$'s are big vectors in~Eq(\ref{column matrix}) for all $i$'s, then we get that $\sum_{j=1}^t x^i_j=\OO(s_A/\psi)$, which implies that $\|\vex^i\|_\infty\le\|\vex^i\|_1=\OO(s_A/\psi)$. By Theorem~\ref{conj:nfold-additive}, there exists an algorithm of running time $2^{2^{\OO(s_D(s_A/\varepsilon)^{t^{\OO(1)}})}}\cdot\textnormal{poly}(|I|)={f(s_A,s_D,t,\varepsilon)}\cdot\textnormal{poly}(|I|)$ which returns a near-feasible solution $\tilde{\vex}$ such that $\tilde{\vex}\in\{\vex: (1-\varepsilon)\veb^0\le\sum_{i=1}^nD^i\vex^i\le (1+\varepsilon)\veb^0, \,A^i\vex^i= \veb^i,\, \ve0\le \vex\le \veu, \vex\in \Z^{nt} \}$, and $\vew\tilde{\vex}\le \OPT$. Theorem~\ref{conj-coro:nfold-additive} follows.



Hence, from now on we assume
there exists some small vector $A^i_j$ in~Eq(\ref{column matrix}) for some $i$. In this case, for every $A^i_j$ we multiply it with some $\lambda_j^i$ to make it big. More precisely, Eq(\ref{column matrix}) 
can be written as
\begin{eqnarray*}\label{column matrix1}
A^i_jx^i_j=\lambda^i_jA^i_j\left\lfloor\frac{x^i_j}{\lambda^i_j}\right\rfloor+A^i_j\left(x^i_j-\lambda^i_j\left\lfloor\frac{x^i_j}{\lambda^i_j}\right\rfloor\right),
\end{eqnarray*}
where the value of $\lambda^i_j\ge 1,\lambda^i_j\in \Z$ is chosen to make the values of some components of $\lambda^i_jA^i_j$ no less than $\psi$, i.e., $\lambda^i_j$ is the smallest integer such that $\lambda^i_jA^i_j$ is big.  

Let $(x^i_j)'\coloneqq\lfloor\frac{x^i_j}{\lambda^i_j}\rfloor$, and $(x^i_j)''\coloneqq x^i_j-\lambda^i_j\lfloor\frac{x^i_j}{\lambda^i_j}\rfloor$. 
Then we have 
\begin{eqnarray*}\label{column matrix2}
A^i_jx^i_j=\lambda^i_jA^i_j(x^i_j)'+A^i_j(x^i_j)''.
\end{eqnarray*}

We have the following straightforward observation.

\begin{observation}\label{obs:split}
\begin{itemize}
\item $(x^i_j)''<\lambda^i_j$;
\item $(x^i_j)''=0$ if $\lambda^i_j=1$;
\item If $(x^i_j)''>0$, then all coordinates of $A^i_j$ are less than $\psi$.
\item If $(x^i_j)''>0$, then $\|\lambda^i_jA^i_j\|_{\infty}\le 2\psi$.
\end{itemize}
\end{observation}


 

Similarly, let $D^i_j$ be the $j$-th column of matrix $D^i$. The constraint $\sum_{i=1}^nD^i\vex^i=\veb^0$ can be rewritten as
\begin{eqnarray*}\label{column matrix3}
\sum_{i=1}^n\sum_{j=1}^t D^i_jx^i_j= \veb^0,
\end{eqnarray*}
where $D^i_jx^i_j$ can be further rewritten as
\begin{eqnarray*}
D^i_jx^i_j=\lambda^i_jD^i_j(x^i_j)'+D^i_j(x^i_j)''.
\end{eqnarray*}
Note that the two vectors $\lambda^i_jD^i_j$ and $D^i_j$ are not necessarily big (or small). Now we can rewrite the input IP as follows. 
\begin{subequations}
 	\begin{eqnarray}
 	(\text{IP}_5)  \quad  \min && \sum_{i=1}^n\sum_{j=1}^t w^i_j\big((x^i_j)'+(x^i_j)''\big)\label{mip:05}\\
 	&&\sum_{i=1}^n\sum_{j=1}^t \Big(\lambda^i_jD^i_j(x^i_j)'+D^i_j(x^i_j)''\Big)= \veb^0\label{IP5_10b}\\
 &&\sum_{j=1}^t \Big(\lambda^i_jA^i_j(x^i_j)'+A^i_j(x^i_j)''\Big)=\veb^i, \quad \forall  1\le i\le n\label{IP5_10c}\\
 &&(x^i_j)'\in[0,(u^i_j)'], (x^i_j)''\in[0,(u^i_j)''], 	\quad\forall  1\le i\le n, 1\le j\le t\nonumber\\	
 && (x^i_j)'\in\Z, (x^i_j)''\in\Z, \quad\forall  1\le i\le n, 1\le j\le t \nonumber
 	\end{eqnarray}
 \end{subequations}
Recall that $0\le x^i_j\le u^i_j$, so here $(u^i_j)'\coloneqq \lfloor\frac{u^i_j}{\lambda^i_j}\rfloor$ and $(u^i_j)''\coloneqq u^i_j-\lambda^i_j\lfloor\frac{u^i_j}{\lambda^i_j}\rfloor$. 

\paragraph*{Step 1: Transform $(\text{IP}_5)$ to a mixed IP with $\OO(1)$ integer variables.} 


Towards the transformation, we first replace the constraint $A^i\vex^i=\veb^i$ with something similar to $\vex^i\in \mathcal{P}^i$.

We first consider the ``minor part'', $(x^i_j)''$. By Observation~\ref{obs:split}, we have that
$$\sum_{j=1}^t A^i_j(x^i_j)''=\sum_{j: \lambda_j^i \neq 1} A^i_j(x^i_j)''+\sum_{j:\lambda_j^i = 1} A^i_j(x^i_j)''= \sum_{j: \lambda_j^i \neq 1} A^i_j(x^i_j)'' .$$  
Since $(x^i_j)''<\lambda_j^i$,  it always holds that
$$\sum_{j: \lambda_j^i \neq 1} A^i_j(x^i_j)''< \sum_{j: \lambda_j^i \neq 1} A^i_j\lambda_j^i \le t \cdot 2(\psi,\psi,\dots,\psi)^\top=2t(\frac{\varepsilon}{2t},\frac{\varepsilon}{2t},\dots,\frac{\varepsilon}{2t})^\top=(\varepsilon,\varepsilon,\dots,\varepsilon)^\top.$$ This means, whatever the value of $(x^i_j)''$ is, its contribution to $A^i\vex^i$ is negligible. 


Next, we consider the ``major part'', $(x^i_j)'$. The fact that $\veb^i\in \{0,1\}^{s_A}$ and $\lambda_j^iA_j^i$ is big guarantees that $(\vex^i)'$ can only take a few distinct values. More precisely, 
it is easy to verify that $\|(\vex^i)'\|_1=\OO(s_A/\psi)$, and hence we can obtain all possible feasible solutions to $(1-\varepsilon)\veb^i\le A^i\vex^i\le (1+\varepsilon)\veb^i$ by a straightforward enumeration in $\OO(s_A/\psi)^t=\OO(s_At/\varepsilon)^t$ time. Denote by $\mathcal{P}^i$ the set of these solutions.

Now it is easy to see that the constraint $A^i\vex^i=\veb^i$ can be safely replaced by $(\vex^i)'\in \mathcal{P}^i$.

Next, we deal with Constraint~(\ref{IP5_10b}): $\sum_{i=1}^n\sum_{j=1}^t \Big(\lambda^i_jD^i_j(x^i_j)'+D^i_j(x^i_j)''\Big)= \veb^0$. We shall deal with the two summands separately. The way we deal with $\sum_{i=1}^n\sum_{j=1}^t \lambda^i_jD^i_j(x^i_j)'$ is the same as Step 1 of Theorem~\ref{conj:nfold-additive}, and the way we deal with $\sum_{i=1}^n\sum_{j=1}^t D^i_j(x^i_j)''$ is the same as Step 1 of Theorem~\ref{conj:IP-general}.


\begin{description}
\item[Step 1.1: Deal with $\sum_{i=1}^n\sum_{j=1}^t \lambda^i_jD^i_j(x^i_j)'$:]
\end{description}

Let $\kappa\coloneqq\|(\vex^i)'\|_\infty=\OO(s_A/\psi)$. 
Recall that $(\vex^i)'\in \mathcal{P}^i$ and let
$\tau\coloneqq\max_{i\in[n]}|\mathcal{P}^i|=\kappa^t$. 
By adding dummy vectors, we may assume $|\mathcal{P}^i|=\tau$ for all $i$. Let $\mathcal{P}^i=\{(\vex^i_1)',(\vex^i_2)',\dots,(\vex^i_{\tau})'\}$. Note that from now on every $(\vex^i_j)'$ is a fixed value rather than a variable. Define $\mathcal{D}^i=\big((D^i)'(\vex^i_1)',(D^i)'(\vex^i_2)',\dots,(D^i)'(\vex^i_{\tau})'\big)$, where $(D^i)'=(\lambda^i_1D^i_1,\lambda^i_2D^i_2,\dots,\lambda^i_t D^i_t).$ The following is essentially the same as Step 1 of Theorem~\ref{conj:nfold-additive} except that the parameters may take different values. For the completeness of this paper, we present all the details.

Let $\Delta_1=\max_{i\in[n]}\|(D^i)'\|_\infty$. It is clear that $\|\mathcal{D}^i\|_\infty=\OO(s_A/\phi)\Delta_1=\OO(s_At\Delta_1/\varepsilon)$. We cut the box $[-\|\mathcal{D}^i\|_\infty,\|\mathcal{D}^i\|_\infty]^{s_D}$ into $(\frac{2}{\delta_1})^{s_D}$ small $\delta_1$-boxes of the form $\prod_{i=1}^{s_D} [(\lambda_i-1)\delta_1\|\mathcal{D}^i\|_\infty,\lambda_i\delta_1\|\mathcal{D}^i\|_\infty]$, where we shall pick a sufficiently small $\delta_1=\frac{\varepsilon}{s_D t\kappa^{t+1}}=\frac{1}{s_D}(\frac{\varepsilon}{s_At})^{\OO(t)}$. 
All the $(D^i)'(\vex^i_\phi)'$'s that are in the same small box are called similar. 

We pick a fixed vector within each $\delta_1$-box and let it be the canonical vector of this $\delta_1$-box. Each $(D^i)'(\vex^i_\phi)'$ lies in some $\delta_1$-box, and we say it corresponds to the canonical vector of this $\delta_1$-box. Hence, we may write $\mathcal{D}^i=\mathcal{V}^i+\tilde{\mathcal{D}}^i$ where $\mathcal{V}^i$ is the matrix consisting of the canonical vectors corresponding to the column vectors of $\mathcal{D}^i$, and $\tilde{\mathcal{D}}^i$ is the residue matrix satisfying that $\|\tilde{\mathcal{D}}^i\|_\infty\le \delta_1\|\mathcal{D}^i\|_\infty=\OO(\delta_1s_At\Delta_1/\varepsilon)$. Furthermore, 
 we make two observations: 

 (i). There are $\rho\coloneqq((\frac{2}{\delta_1})^{s_D})^{\tau}=(\frac{2}{\delta_1})^{s_D\tau}$ different types of canonical matrix $\mathcal{V}^i$'s since each $\mathcal{V}^i$ has $\tau$ columns. For $k=1,\ldots,\rho$, let $I_k$ denote the set of indices of $\mathcal{D}^i$'s whose canonical matrix is of type $k$. Denote by $(\vev^k_1,\vev^k_2,\dots,\vev^k_{\tau})$ the canonical matrix of type $k$. Then if the canonical matrix of $\mathcal{D}^i$ is of type $k$, we have that 
 $$\tilde{\mathcal{D}}^i=\mathcal{D}^i-\mathcal{V}^i=\big((D^i)'(\vex^i_1)'-\vev^k_1,(D^i)'(\vex^i_2)'-\vev^k_2,\dots,(D^i)'(\vex^i_\tau)'-\vev^k_{\tau}\big).$$  
 
 (ii). With $(\vex^i)'\in\mathcal{P}^i$, the constraint $\sum_{i=1}^n(D^i)'(\vex^i)'$ can be interpreted as selecting some column from each $\mathcal{D}^i$. 
 Hence, we introduce a binary variable $z_{i\phi}$ such that $z_{i\phi}=1$ denotes that we select the $\phi$-th column of $\mathcal{D}^i$, and $z_{i\phi}=0$ otherwise. Notice that for $i\in I_k$, $\mathcal{D}^i$ corresponds to the canonical matrix of type $k$. We define $y_{k\phi}\coloneqq\sum_{i\in I_k}z_{i\phi}$, where $y_{k\phi}$ can be interpreted as counting how many times the $\phi$-th column of the type-$k$ canonical matrix are selected. With the new variables, we may rewrite the constraint as:
\begin{eqnarray}
\sum_{i=1}^n\sum_{j=1}^t \lambda^i_jD^i_j(x^i_j)'=\sum_{k=1}^{\rho}\sum_{\phi=1}^\tau \vev_\phi^k y_{k\phi}+ \sum_{k=1}^{\rho}\sum_{i\in I_k}\sum_{\phi=1}^\tau ((D^i)'(\vex^i_\phi)'-\vev_\phi^k)z_{i\phi},\label{equality_11}
\end{eqnarray}
together with two extra constraints: 
\begin{eqnarray}
y_{k\phi}=\sum_{i\in I_k}z_{i\phi}, \quad \forall 1\le \phi\le \tau, 1\le k\le \rho,\label{equality_12}
\end{eqnarray}
and
\begin{eqnarray}
\sum_{\phi=1}^{\tau}z_{i\phi}=1, \quad \forall  i\in I_k, 1\le k\le \rho.\label{equality_13}
\end{eqnarray}
Following this notation, $\sum_{i=1}^n\sum_{j=1}^t w^i_j(x^i_j)'$ in (\ref{mip:05}) can be rewritten as $\sum_{i=1}^{n}(\sum_{\phi=1}^{\tau} \vew^i(\vex^i_\phi)'z_{i\phi})$.

\begin{description}
\item[Step 1.2: Deal with $\sum_{i=1}^n\sum_{j=1}^t D^i_j(x^i_j)''$:]
\end{description}

 This part is essentially the same as Step 1 of Theorem~\ref{conj:IP-general}, except that the parameters may take different values. We present the details for completeness.

Let $\Delta_2=\max_{i\in[n]}\|D^i\|_\infty$, then $D^i_j\in \interval {-\Delta_2}{\Delta_2}^{s_D}$.
We cut $\interval{-\Delta_2}{\Delta_2}^{s_D}$ into $(2/\delta_2)^{s_D}$ small boxes (called $\delta_2$-boxes), where $\delta_2=\OO(\frac{\varepsilon}{s_D})$. Each $\delta_2$-box is of the form $\prod_{i=1}^{s_D} \interval {(\lambda_i-1)\delta_2\Delta_2}{\lambda_i\delta_2\Delta_2}$, where $\lambda_i\in\{-1/\delta_2+1,-1/\delta_2+2,\dots,1/\delta_2\}$.  We arbitrarily index all the $\delta_2$-boxes, and let $I_d$ be the set of indices $(i,j)$'s of $D^i_j$'s that are in the $d$-th $\delta_2$-box, where $d=1,\ldots, (2/\delta_2)^{s_D}$. 
 

For the $d$-th $\delta_2$-box, we arbitrarily pick one vector $\vev_d$ as its fixed canonical vector. For each $(i,j)\in I_d$, we define $\tilde{D}_j^i=D^i_j-\vev_d$, and then it is clear that $\|\tilde{D}_j^i\|_{\infty}\le \delta_2\Delta_2$.  
Now we can rewrite $\sum_{i=1}^n\sum_{j=1}^t D^i_j(x^i_j)''$ as follows by introducing a new variable 
\begin{eqnarray}
y_d \coloneqq \sum_{(i,j)\in I_d}(x^i_j)'', \quad \forall 1\le d\le (2/\delta_2)^{s_D},\label{equality_15}
\end{eqnarray}
and thus
 \begin{eqnarray}
\sum_{i=1}^n\sum_{j=1}^t D^i_j(x^i_j)''=\sum_{d=1}^{(2/\delta_2)^{s_D}}\sum_{(i,j)\in I_d}(\vev_d+\tilde{D}_j^i)(x^i_j)''=\sum_{d=1}^{(2/\delta_2)^{s_D}}\vev_dy_d+\sum_{d=1}^{(2/\delta_2)^{s_D}}\sum_{(i,j)\in I_d} \tilde{D}_j^i(x^i_j)'',\label{equality_14}
\end{eqnarray}

 
 Now we are ready to reformulate $(\text{IP}_5)$ as a mixed IP. Combining the equations~(\ref{equality_11}),~(\ref{equality_12}),~(\ref{equality_13}), (\ref{equality_15}),~(\ref{equality_14}) with the objective function, we have
 \begin{subequations}
\begin{eqnarray}
 	(\text{MIP}_6)  \quad  \min && \sum_{i=1}^{n}(\sum_{\phi=1}^{\tau} \vew^i(\vex^i_\phi)'z_{i\phi})+\sum_{i=1}^n\sum_{j=1}^t w^i_j(x^i_j)''\\
&&\sum_{k=1}^{\rho}\sum_{\phi=1}^\tau \vev_\phi^k y_{k\phi}+ \sum_{k=1}^{\rho}\sum_{i\in I_k}\sum_{\phi=1}^\tau ((D^i)'(\vex^i_\phi)'-\vev_\phi^k)z_{i\phi}\nonumber\\
&&+\sum_{d=1}^{(2/\delta_2)^{s_D}}\vev_dy_d+\sum_{d=1}^{(2/\delta_2)^{s_D}}\sum_{(i,j)\in I_d} \tilde{D}_j^i(x^i_j)''=\veb^0\label{IP6_10b}\\
 && y_{k\phi}=\sum_{i\in I_k}z_{i\phi}, \quad \forall 1\le \phi\le \tau, 1\le k\le \rho\label{IP6_10c}\\
 &&\sum_{\phi=1}^{\tau}z_{i\phi}=1, \quad \forall  i\in I_k, 1\le k\le \rho\label{IP6_10d}\\
 &&\sum_{(i,j)\in I_d}(x^i_j)''=y_d, \quad \forall 1\le d\le (2/\delta_2)^{s_D}\\
 &&y_{k\phi}\in\Z, z_{i\phi}\in[0,1],  	\quad \forall  1\le \phi\le \tau, i\in I_k, 1\le k\le \rho\nonumber\\
	&&y_d\in \Z, (x^i_j)''\in[0,(u^i_j)''], \quad\forall  (i,j)\in I_d, 1\le d\le (2/\delta_2)^{s_D}\nonumber
 	\end{eqnarray}
\end{subequations}
Observe that $(\text{MIP}_6)$ relaxes $z_{i\phi}\in\{0,1\}$ to $z_{i\phi}\in[0,1]$ and removes the constraint $(x^i_j)''\in\Z$. Thus, it only contains $\rho\tau$ integer variables $y_{k\phi}$, and $(2/\delta_2)^{s_D}$ integer variables $y_d$. Thus, applying Kannan's algorithm~\cite{kannan1987minkowski}, an optimal solution to $(\text{MIP}_6)$ can be computed in $(\rho\tau+(2/\delta_2)^{s_D})^{\OO(\rho\tau+(2/\delta_2)^{s_D})}\cdot |I|$ time. Recall that $\rho=(\frac{2}{\delta_1})^{s_D\tau}$, $\tau=\OO(s_At/\varepsilon)^t$, $\delta_1=\frac{1}{s_D}(\frac{\varepsilon}{s_At})^{\OO(t)}$, and $\delta_2=\OO(\frac{\varepsilon}{s_D})$, so $(\text{MIP}_6)$ can be computed in $2^{\big({2^{{(s_D\cdot ({s_At}/{\varepsilon})^t)^{\OO(1)}}}+({s_D}/{\varepsilon})^{\OO(s_D)}\big)}^{\OO(1)}}\cdot |I|=f(s_A,s_D,t,\varepsilon)\cdot \textnormal{poly}(|I|)$ time.  

Let $y_{k\phi}=y_{k\phi}^*$, $y_d=y_d^*$, $(x^i_j)''=(x^i_j)''^*$, and $z_{i\phi}=z_{i\phi}^*$ be the optimal solution to $(\text{MIP}_6)$. Note that $y_{k\phi}^*\in\Z$, $y_d^*\in\Z$. By fixing those integer variables, $(\text{MIP}_6)$ becomes an LP with its optimal solution being $(x^i_j)''^*\in \R$, and $z_{i\phi}^*\in [0,1]$. The following step is devoted to rounding these variables. 

\paragraph*{Step 2: Obtain a feasible solution with $\OO(s_D\tau)$ variables taking a fractional value.}

When we round the fractional variables, we can decompose $(\text{MIP}_6)$ into the following two separate LPs:
\begin{subequations}
\begin{eqnarray}
 	(\text{LP}_7)  \quad  \min  && \sum_{i=1}^{n}(\sum_{\phi=1}^{\tau} \vew^i(\vex^i_\phi)'z_{i\phi})\nonumber\\
&&\sum_{k=1}^{\rho}\sum_{\phi=1}^\tau \vev_\phi^k y_{k\phi}+ \sum_{k=1}^{\rho}\sum_{i\in I_k}\sum_{\phi=1}^\tau ((D^i)'(\vex^i_\phi)'-\vev_\phi^k)z_{i\phi}\nonumber\\
&&= \sum_{k=1}^{\rho}\sum_{\phi=1}^\tau \vev_\phi^k y_{k\phi}^*+ \sum_{k=1}^{\rho}\sum_{i\in I_k}\sum_{\phi=1}^\tau ((D^i)'(\vex^i_\phi)'-\vev_\phi^k)z_{i\phi}^*\nonumber\\
 && \sum_{i\in I_k}z_{i\phi}=y_{k\phi}^*, \quad \forall 1\le \phi\le \tau, 1\le k\le \rho\nonumber\\
 &&\sum_{\phi=1}^{\tau}z_{i\phi}=1, \quad \forall  i\in I_k, 1\le k\le \rho\nonumber\\
 && z_{i\phi}\in[0,1],  	\quad \forall  1\le \phi\le \tau, i\in I_k, 1\le k\le \rho\nonumber
 	\end{eqnarray}
\end{subequations}
 \begin{subequations}
\begin{eqnarray*}
 	(\text{LP}_8)  \quad \min  && \sum_{i=1}^n\sum_{j=1}^t w^i_j(x^i_j)''\\
&&\sum_{d=1}^{(2/\delta_2)^{s_D}}\sum_{(i,j)\in I_d} \tilde{D}_j^i(x^i_j)''=\sum_{d=1}^{(2/\delta_2)^{s_D}}\sum_{(i,j)\in I_d} \tilde{D}_j^i(x^i_j)''^*\\
 &&\sum_{(i,j)\in I_d}(x^i_j)''=y_d^*, \quad \forall 1\le d\le (2/\delta_2)^{s_D}\\
	&&(x^i_j)''\in[0,(u^i_j)''], \quad\forall  (i,j)\in I_d, 1\le d\le (2/\delta_2)^{s_D}
 	\end{eqnarray*}
\end{subequations}

Observe that $(\text{LP}_7)$ and $(\text{LP}_8)$ satisfy the structure of the LP obtained in Step 2 of  Theorem~\ref{conj:nfold-additive} and Theorem~\ref{conj:IP-general} respectively, so we can directly apply the rounding procedures there.


\paragraph*{Step 3: Round these $\OO(s_D\tau)$ fractional variables.}

 The rounding procedure in this step also follows straightforwardly from applying Step 3 of Theorem~\ref{conj:nfold-additive} and Theorem~\ref{conj:IP-general} to $(\text{LP}_7)$ and $(\text{LP}_8)$, since the two parts $z_{i\phi}$ and $(x^i_j)''$ are  relatively independent. Hence, it brings a total error of $\OO(s_D\tau)\cdot \max\{\delta_1 \|{\mathcal{D}}^i\|_\infty,\delta_2 \Delta_2\}$ to Constraint~(\ref{IP6_10b}).  
Note that $\OO(s_D\tau)\cdot \max\{\delta_1 \|{\mathcal{D}}^i\|_\infty,\delta_2 \Delta_2\}\le \OO(s_D\tau)\cdot\max\{\delta_1,\delta_2\}\cdot b^0_j$ in the $j$-th dimension. We can take $\delta_1=\frac{1}{s_D}(\frac{\varepsilon}{s_At})^{\OO(t)},\delta_2=\OO(\frac{\varepsilon}{s_D})$ to make $\OO(s_D\tau)\cdot\max\{\delta_1,\delta_2\}\le \varepsilon$. Therefore, we have $(1-\varepsilon)\veb^0\le\sum_{i=1}^nD^i\vex^i\le (1+\varepsilon)\veb^0$.

\begin{description}
    \item[Running time.] 
\end{description}

Step 1: $2^{{2^{{({s_D}/{\varepsilon})^{\OO(s_D)}\cdot ({s_At}/{\varepsilon})^{\OO(t)}}}}}\cdot |I|$


Step 2: $(\rho n)^4\cdot |I|=2^{{(s_D\cdot ({s_At}/{\varepsilon})^t)^{\OO(1)}}} n^4\cdot |I|$

Step 3: $\rho\tau^2 n^2\log n$+$(2/\delta_2)^{s_D}\cdot n\log n=(2^{{(s_D\cdot ({s_At}/{\varepsilon})^t)^{\OO(1)}}}+({s_D}/{\varepsilon})^{\OO(s_D)})\cdot n^2\log n$

All in all, the total time is $2^{{2^{{({s_D}/{\varepsilon})^{\OO(s_D)}\cdot ({s_At}/{\varepsilon})^{\OO(t)}}}}}\cdot \textnormal{poly}(|I|)$. 
\end{proof}

\end{document}